\documentclass[12pt]{amsart}

\usepackage{latexsym, amssymb, amsmath,ulem,soul,esint}
\usepackage{tikz}

\usepackage{amsfonts, graphicx}
\usepackage{graphicx,color}
\newcommand{\be}{\begin{equation}}
\newcommand{\ee}{\end{equation}}
\newcommand{\beq}{\begin{eqnarray}}
\newcommand{\eeq}{\end{eqnarray}}

\newtheorem{thm}{Theorem}[section]

\newtheorem{lma}{Lemma}[section]
\newtheorem{prop}{Proposition}[section]

\usepackage{mathrsfs} 

\usepackage{wasysym,stmaryrd}
\newtheorem{claim}{Claim}[section]
\theoremstyle{remark}
\newtheorem{rem}{Remark}[section]
\numberwithin{equation}{section}

\def\be{\begin{equation}}
\def\ee{\end{equation}}
\def\bee{\begin{equation*}}
\def\eee{\end{equation*}}

\def\lf{\left}
\def\ri{\right}

\def\K{K\"ahler }

\def\KR{K\"ahler-Ricci }
\def\Ric{\text{\rm Ric}}
\def\Rm{\text{\rm Rm}}

\def\p{\partial}

\def\rheat{\lf(\frac{\p}{\p t}-\Delta_{g(t)}\ri)}

\def\t{\tilde}

\def\e{\varepsilon}

\def\a{{\alpha}}
\def\b{{\beta}}

\def\R{\mathbb{R}}

\def\V{\mathrm{Vol}}

\def\ddb{\sqrt{-1}\partial\bar\partial}

\begin{document}

\title[]
{Gap Theorem on Riemannian manifolds using Ricci flow}

\author{Pak-Yeung Chan}
\address[Pak-Yeung Chan]{Department of Mathematics, University of California, San Diego, La Jolla, CA 92093
}
\email{pachan@ucsd.edu}

\author{Man-Chun Lee}
\address[Man-Chun Lee]{Department of Mathematics, The Chinese University of Hong Kong, Shatin, N.T., Hong Kong
}
\email{mclee@math.cuhk.edu.hk}

\renewcommand{\subjclassname}{
  \textup{2020} Mathematics Subject Classification}

\date{\today}

\begin{abstract}
In this work, we use the Ricci flow approach to study the gap phenomenon of Riemannian manifolds with non-negative curvature and sub-critical scaling invariant curvature decay.  
The first main result is a quantitative Ricci flow existence theory without non-collapsing assumption. We use it to show that complete non-compact manifolds with non-negative complex sectional curvature and sufficiently small average curvature decay are necessarily flat. The second main result concerns three-manifolds with non-negative Ricci curvature of quadratic decay. By combining our newly established curvature estimate and method in \K geometry, we show that if the curvature decays slightly faster even in average sense,  the manifold must be flat. This strengthens a result of Reiris. In the compact case, we  use the Ricci flow regularization to generalize the celebrated Gromov-Ruh Theorem in this direction.
\end{abstract}


\maketitle

\markboth{Pak-Yeung Chan and Man-Chun Lee}{}

\section{Introduction}
Let $(M^n,g)$ be a complete Riemannian manifold.  The purpose of this work is to study the geometric quantity
$$k(x,r)=r^2 \fint_{B_{g_0}(x,r)}|\Rm(g_0)|\, d\mathrm{vol}_{g_0}$$
on manifolds with non-negative or almost non-negative curvature.  The quantity is natural in the sense that it is scaling invariant and is sometimes regarded as the $L^1$ version of Morrey bound on curvature.  We consider the case where volume is rescaled so that $k(\cdot,r)$ measures the flatness in an average sense. In \K geometry, it has been studied extensively and is deeply related to the function theory, for example see \cite{NiTam,NiShiTam}. 

The geometric significance 
of $k(\cdot ,r)$ is on two-folds. In the non-compact case, when $r$ tends to infinity, it measures the complexity of the infinity in an averaging sense and is related to the gap phenomenon of flat metric. When $M$ is complete non-compact with non-negative curvature, we are interested in asking how much positive curvature such a manifold could have. 
By the classical theorem of  Bonnet-Meyer,  it is clear that the curvature can't be too positive point-wise otherwise the manifold will be closing up at infinity contradicting the non-compactness.  One might ask to what extent  a non-flat metric will stay away from the flat metric at infinity. In dimension two,  it is not difficult to construct metric with zero curvature outside compact set, non-negative curvature everywhere, and positive curvature somewhere while in higher dimensions, 
the situation is different.  The first result of this type was originated by Mok-Siu-Yau \cite{MokSiuYau1981} where it was proved that when the complex dimension $m\geq 2$,  $M$ is isometrically biholomorphic to the Euclidean space if it has non-negative holomorphic bisectional curvature, Euclidean volume growth and has ``faster-than-quadratic" 
curvature decay.  Shortly after,  Greene-Wu \cite{GreeneWu1982} considered the more general Riemannian case and proved that manifold with a pole, with ``faster-than-quadratic" 
curvature decay  and with non-negative sectional curvature is necessarily flat if the dimension of the manifold is $\geq 3$ except when the dimension is 4 or 8.  The pole assumption was later removed by  Eschenburg-Schroeder-Strake \cite{EschenburgSchroederStrake1989} and Dress \cite{Dress}.  See also the relatively more recent work of Greene-Petersen-Zhu \cite{GreenePetersenZhu1994}.   

On the other hand,  in the \K case,  effort has been made to  improve the gap in Theorem of Mok-Siu-Yau. In \cite{Ni2012},  Ni found the optimal gap in term of $k(\cdot,r)$ which states that a complete non-compact \K manifold must be flat if it has  non-negative holomorphic bisectional curvature and $k(x_0,r)=o(1)$ as $r\to +\infty$ for some $x_0\in M$.  See also \cite{NiNiu2020,ChenZhu2002,Li2016} for more related works.  The method employed by Ni \cite{Ni2012} is based on finding the Ricci potential via  solving the Poincar\'e-Lelong equation on non-compact \K manifolds.  With the sub-critical decay rate of $k(o,r)$ as $r\to +\infty$ for some $o\in M$. One can find $u\in C^\infty_{loc}(M)$ such that $\ddb u=\Ric(g_0)$ and $u=o(\log r)$.  Hence, the flatness of $g_0$ follows from a  Liouville Theorem for pluri-subharmonic  on non-negatively curved \K manifolds. Thus, the K\"ahler structure of the manifold plays an important role. 

In this work we are interested in its Riemannian analogy. One first needs to identify a suitable non-negativity of curvature.  Our first result is partially motivated by the parallel theory of \KR flow. We consider the complex sectional curvature in general dimension.   To clarify the notation, we say that an 
algebraic curvature tensor 
$\mathrm{R}$ has non-negative complex sectional curvature $\mathrm{K}^\mathbb{C}(\mathrm{R})\geq 0$ if to each two-complex-dimensional subspace $\Sigma$,  the complexified $\mathrm{R}$ satisfies $\mathrm{R}(u,v,\bar u,\bar v)\geq 0$ for all unitary basis $\{u,v\}$ of $\Sigma$.  If one instead asks for non-negativity of complex sectional curvature only for PIC1 sections, defined to be those $\Sigma$ that contain some non-zero vector $v$ whose conjugate $\bar v$ is orthogonal to $\Sigma$, then we say that $\mathrm{R}\in \mathrm{C}_{\mathrm{PIC1}}$.  When $n=3$,  non-negative complex sectional curvature is equivalent to non-negative sectional curvature while $\mathrm{R}\in \mathrm{C}_{\mathrm{PIC1}}$ is equivalent to $\Ric\geq 0$. When $n\geq 4$, one can equivalently describe $\mathrm{K}^\mathbb{C}(\mathrm{R})\geq 0$ by requiring 
$$\mathrm{R}_{1331}+\lambda^2 \mathrm{R}_{1441}+\mu^2 \mathrm{R}_{2332}+\lambda^2 \mu^2 \mathrm{R}_{2442}+2\lambda \mu \mathrm{R}_{1234}\ge 0$$ 
for any orthonormal four-frames $\{e_i\}_{i=1}^4$ and $\lambda,\mu\in [0,1]$. Similarly,  $\mathrm{R}\in \mathrm{C}_{\mathrm{PIC1}}$ if 
$$\mathrm{R}_{1331}+\lambda^2 \mathrm{R}_{1441}+ \mathrm{R}_{2332}+\lambda^2  \mathrm{R}_{2442}+2\lambda  \mathrm{R}_{1234}\ge 0$$ 
for any orthonormal four-frames $\{e_i\}_{i=1}^4$ and $\lambda\in [0,1]$. We refer interested readers to the book by Brendle \cite{BrendleBook} for an overview on their importance in differentiable sphere Theorem.   Motivated by Ni's optimal gap Theorem in the %
\K case,  we have the following gap Theorem in the Riemannian case under an asymptotic condition of $k(\cdot, r)$ and \textbf{without} volume growth assumption. 

\begin{thm}\label{Thm:GapTheorem-higherD} 
There exists $\e_0(n)>0$ such that the following holds: If $(M^n,g_0)$ is a complete non-compact manifold such that $n\geq 3$,  $\mathrm{K}^\mathbb{C}(g_0)\geq 0$ and
       \begin{equation}
\int^{+\infty}_0 s\left(\fint_{B_{g_0}(x,s)}|\Rm(g_0)|\, d\mathrm{vol}_{g_0} \right) ds<\e_0
       \end{equation}
       for all $x\in M$. Then $(M,g_0)$ is flat.
\end{thm}

 We would like to stress that the gap here is scaling invariant and can be computed explicitly. Moreover, $g_0$ is not necessarily of Euclidean volume growth.  In particular, the Ricci flow theory and compactness theory are not well-understood in this setting. To the best of our knowledge,  this seems to be the first gap Theorem in Riemannian case under condition in $k(\cdot,r)$. In contrast with the \K case by Ni \cite{Ni2012}, the asymptotic assumption on $k(\cdot,r)$ is stronger. The \K structure is however unnecessary and in particular, attacking the question using the function theory is in fact unavailable at the present. Our approach is largely motivated by a gap Theorem of Chen-Zhu \cite{ChenZhu2003} which relies on using the long-time asymptotic of \KR flow. The strategy we use is to deform the given metric by Ricci flow for all time and to analyse its long-time asymptotic behaviour. By showing that the blow-down Ricci flow is asymptotically flat, one can prove 
that the metric is initially flat using Brendle's Harnack inequality \cite{Brendle2009} in case of $\mathrm{K}^\mathbb{C}(g(t))\geq 0$ or Colding's volume convergence \cite{CheegerColding1997} in case of Euclidean volume growth. 
To implement the strategy,  the fundamental difficulty
is to run the flow for all time.  Although the metric is expected to be flat or close to be flat under suitable average curvature decay condition, the manifold can still be a-priori very complicated at infinity. In particular, there is no general existence theory on producing the Ricci flow even for a short-time. Although a theory for Ricci flow starting with general complete manifolds with non-negative complex sectional curvature was developed by Cabezas-Rivas and Wilking \cite{CW}, we need to develop one with quantitative estimates. And more importantly, we need to find a way to extract curvature estimate from the \textit{initial} asymptotic behaviour of $k(\cdot,r)$, both in $r\to 0$ and $r\to +\infty$. In this regard, we obtain a heat kernel estimate for the operator $\partial_t-\Delta_{g(t)}-\mathcal{R}_{g(t)}$. The heat kernel estimate was previously studied in \cite{CW} for \textbf{non-collapsing} Ricci flows with scaling invariant curvature decay.   In this work, we combine the ideas in \cite{CW} with the idea of the relative volume comparison Theorem in  \cite{TianZhang2021} to establish a heat kernel estimate which \textbf{does not} rely on volume non-collapsing. Using this, we establish a quantitative short-time existence theory which \textit{only} relies on $k(\cdot, r)$ and a lower bound of curvature. For notation convenience, throughout this work we will use $a\wedge b$ to denote $\min\{a,b\}$ for $a,b\in \mathbb{R}$. 

\begin{thm}\label{Thm:RF-existence}
For any $n\geq 3$, there exist 
$\e_0(n),\a_n, S_n,L_n>0$ such that the following holds: Suppose $(M,g_0)$ is a complete manifold and for some $r>0$ and $\Lambda_0\in (0,1)$, the initial metric $g_0$ satisfies 
    \begin{enumerate}
    \item[(a)] $\inf_M K(g_0)>-\infty$; 
        \item[(b)] 
       \begin{enumerate}
           \item[(i)] $\mathrm{Rm}(g_0)+\Lambda_0r^{-2} g_0\owedge g_0/2
           \in \mathrm{C}_{\mathrm{PIC1}}$ if $n\geq 4$; or
           \item[(ii)] $K(g_0)+\Lambda_0 r^{-2}\geq 0$ if $n=3$,
       \end{enumerate} 
       \item[(c)] for all $x\in M$,
       \begin{equation}
 \int^{r}_0 s\left(\fint_{B_{g_0}(x,s)}|\Rm(g_0)|\, d\mathrm{vol}_{g_0} \right) ds< \e_0.
       \end{equation}
    \end{enumerate}
    Then there exists a complete short-time solution $g(t)$ to the Ricci flow on $M\times [0,S_n(r^2\wedge \mathrm{diam}(M,g_0)^2)]$ with $g(0)=g_0$ and satisfies 
    \begin{enumerate}
        \item[(I)] $\sup_M |\Rm(g(t))|\leq \a_n t^{-1}$;
        \item [(II)] $\mathrm{Rm}(g(t))+L_n\Lambda_0 r^{-2}g(t)\owedge g(t)/2 \in \mathrm{C}_{\mathrm{PIC1}}$ if $n\geq 4$;
        \item [(III)] $K(g(t))+L_n\Lambda_0 r^{-2}\geq 0$ if $n=3$
    \end{enumerate}
    on $M\times (0,S_n(r^2\wedge \mathrm{diam}(M,g_0)^2)]$.    When $M$ is non-compact, $\mathrm{diam}(M,g_0)$ is understood to be $+\infty$.
\end{thm}

Here $\owedge$ denotes the Kulkarni-Nomizu product and $\frac 12 g\owedge g$ refers to the curvature tensor of standard %
sphere.  The quantitative existence theory also provides us a long-time existence of Ricci flow with estimates under the assumption in Theorem~\ref{Thm:GapTheorem-higherD}, see Proposition~\ref{prop:LongTime}. We require small integral bound on $k(x,r)$ uniform in $x\in M$. This is deeply related to the uniform regularization of the Ricci flow throughout $M$. In fact, the asymptotic of $k(\cdot,r)$ when $r\to 0$ detects how regular the centre is.  In particular, if the curvature is bounded,  $k(\cdot ,r)=O(r^2)$  while the metric cone as a singular model will have $k(o,r)=O(1)$ at the tip $o$.  In this sense, the integrability of $r^{-1}k(\cdot,r)$ at $r=0$ is indeed measuring  asymptotically how flat the tangent cone is. It is also interesting to compare this with Shi's long-time existence criteria of the \KR flow \cite{Shi1997} in the \K case with bounded non-negative bisectional curvature. In comparison with Theorem~\ref{Thm:RF-existence}, Shi showed that in this case, the \KR flow exists for all time with curvature decay in $\a t^{-1}$ for some $\a>0$ if $k(x,r)$ is uniformly bounded for all $x\in M$ and $r>0$. In view of the application of the \KR flow to Yau's uniformization conjecture, it will be important to see to what extent the result of \KR flow can be generalized to Ricci flow.

 On the other hand, it is interesting to compare Theorem~\ref{Thm:GapTheorem-higherD} with the Bryant expanding Ricci soliton with positive curvature operator which has 
Euclidean volume growth and $|\Rm|\leq C(d_g(x,x_0)^2+1)^{-1}$ but is non-flat (see \cite{Bryant2005, ChowBookI}).  Using the method we develop in the proof of Theorem~\ref{Thm:GapTheorem-higherD},  we shows that in the case of Euclidean volume growth, %
if we 
strengthen 
the decay rate slightly even in the average sense, then the manifold is necessarily isometric to the Euclidean space. In this sense, the gap Theorem is optimal. 

\begin{thm}\label{thm:gap-PIC1}
 Suppose $(M^n,g_0)$ is a complete non-compact manifold for $n\geq 3$ such that $g_0$ has Euclidean volume growth and 
\begin{enumerate}
    \item[(i)]$\mathrm{Rm}(g_0)\in \mathrm{C}_{\mathrm{PIC1}}$ if $n\geq 4$ or;
    \item[(ii)] $\mathrm{K}(g_0)\geq 0$ if $n=3$.
\end{enumerate} 
 If there exists $x_0\in M$ such that 
 $$k(x_0,r)=r^2\fint_{B_{g_0}(x_0,r)} |\Rm(g_0)|\,d\mathrm{vol}_{g_0}=o(1)$$
    as $r\to +\infty$, then $(M,g_0)$ is isometric to the flat Euclidean space. 
\end{thm}

Our second main result concerns the three dimension case which is motivated by the work of Reiris.   In \cite{Reiris2015}, he studied the relation between the point-wise quadratic curvature decay and volume non-collapsing.  Among other things, he showed that if $g_0$ is a complete metric on $\mathbb{R}^3$ with quadratic curvature decay, then it must be of Euclidean volume growth.  It is natural to ask what we can say if we strengthen the curvature decay slightly.  At the same time, one might compare the setting with Theorem~\ref{thm:gap-PIC1}.  Even if the $\mathrm{PIC1}$ is a natural generalization of $\Ric\geq 0$ in dimension three, the Ricci flow does not behave in the same way as in the proof of Theorem~\ref{thm:gap-PIC1}. We are able to overcome this by using some special feature in dimension three. By using Liu's splitting Theorem \cite{Liu2013}, method in \K geometry \cite{Ni2012} and the heat kernel estimate, we establish the following:
\begin{thm}\label{thm:3D-gap}
Suppose $(M^3,g_0)$ is complete non-compact manifold such that $\Ric(g_0)\geq 0$ and for some $x_0\in M$, we have 
\begin{equation}
     \limsup_{x\to +\infty} \left( d_{g_0}(x,x_0)^2 \cdot \mathcal{R}(g_0(x))\right)<+\infty \quad \text{and}\quad 
     k(x_0,r)=o(1)
\end{equation}
as $r\to +\infty$. Then $(M^3,g_0)$ is flat.
\end{thm}

Here $\mathcal{R}$ denotes the scalar curvature. By taking the product of $\mathbb{R}$ with a metric in dimension two which is conical at infinity, we can easily see that the asymptotic of $k(x_0,\cdot)$ is optimal.  This strengthens the result of Reiris \cite{Reiris2015} and improves the gap Theorem in \cite{GreeneWu1982,GreenePetersenZhu1994}. The analogous statement in higher dimensions cannot be true unless we strengthen the curvature or topological conditions which can be easily seen from the Eguchi-Hanson metric.

Theorem~\ref{Thm:RF-existence} can be regarded as a result of regularization with possibly collapsing initial data. The situation where the volume is collapsing with almost vanishing curvature has been studied extensively in the past. In particular,  the Gromov–Ruh Theorem \cite{Gromov1978,Ruh1982} states that if one normalizes the diameter of $(M,g)$ to be 1, then $M$ is diffeomorphic to an infranil manifold if its curvature is sufficiently small depending only on some dimensional constant. 
It has recently been generalized by Chen-Wei-Ye \cite{ChenWeiYe2022} to $L^{n/2}$ bound of curvature weighed by squared Sobolev constant. Motivated by their method, we generalize the Gromov–Ruh Theorem in direction using  $k(\cdot,r)$.

\begin{thm}\label{thm:almost-flat}
    For any $n\geq 3$, there exists $\e_0(n)>0$ such that the following holds: A compact manifold $M^n$ is diffeomorphic to an infranil manifold if it admits a metric $g_0$ such that 
    \begin{enumerate}
        \item[(i)] $\mathrm{Rm}(g_0)+ \e_0\cdot \mathrm{diam}(g_0)^{-2} g_0\owedge g_0\in \mathrm{C}_{\mathrm{PIC1}}$ if $n\geq 4$;
        \item[(ii)] $\mathrm{K}(g_0)\geq -\e_0\cdot \mathrm{diam}(g_0)^{-2}$ if $n=3$;
        \item[(iii)] for all $x\in M$,
\begin{equation}
\int^{\mathrm{diam}(M,g_0)}_0 s\left( \fint_{B_{g_0}(x,s)}|\Rm(g_0)|\, d\mathrm{vol}_{g_0}\right)\, ds<\e_0.
       \end{equation}
    \end{enumerate}
\end{thm}

{\it Acknowledgement}: The second named author was partially supported by Hong Kong RGC grant (Early Career Scheme) of Hong Kong No. 24304222, a direct grant of CUHK and a NSFC grant.

\section{Some preliminaries  on Ricci flow}
The novel idea of this work is to regularize the initial metric $g_0$ using Ricci flow. This is a one parameter family of metrics $g(t)$ satisfying 
\begin{equation}
    \left\{
    \begin{array}{cc}
        \partial_tg(t) =-2\Ric(g(t));  \\
        g(0) =g_0 
    \end{array}
    \right.
\end{equation}
In this section, we will collect some preliminaries of Ricci flow technique which will be used throughout this work.

\subsection{Heat kernel estimates}
To obtain curvature estimate, we will make use of the heat kernel coupled with the Ricci flow.  Let $g(t)$ be a complete Ricci flow on $M\times[0,T]$ with $g(0)=g_0$.  Let $\Omega\Subset M$ be an open set with smooth boundary. We let $G(x,t;y,s),t>s$ be the dirichlet heat kernel for the backward heat equation coupled with the Ricci flow $g(t)$: 
\begin{equation}
\left\{
\begin{array}{ll}
\left(\partial_s+\Delta_{y,g(s)}\right)G(x,t;y,s)=0,\quad\text{on}\quad \Omega\times \Omega\times [0,t);\\
\lim_{s\to t^-} G(x,t;y,s)=\delta_x(y), \quad\text{for}\; x\in \Omega;\\
G(x,t;y,s)=0, \quad\text{for}\; x\in \Omega\; \text{and}\; y\in \partial\Omega.
\end{array}
\right.
\end{equation}
Then, 
\begin{equation}
\left\{
\begin{array}{ll}
\left(\partial_t-\Delta_{x,g(t)}-\mathcal{R}(g(x,t))\right)G(x,t;y,s)=0,\quad\text{on}\quad \Omega\times \Omega\times (s,T];\\
\lim_{t\to s^+} G(x,t;y,s)=\delta_y(x), \quad\text{for}\; y\in \Omega;\\
G(x,t;y,s)=0, \quad\text{for}\; y\in \Omega\; \text{and}\; x\in \partial\Omega.
\end{array}
\right.
\end{equation}
Such $G$ exists and is positive in the interior of $\Omega$, see \cite{Guenther2002}. In this work, all heat kernel will be referring to the heat kernel with respect to $\partial_t-\Delta_{g(t)}-\mathcal{R}$ as described above. We have the following heat kernel estimates modified from \cite{LeeTam2022} building on \cite{BamlerCabezasWilking2019}.

\begin{prop}\label{prop: heat Kernel}
For any $n,\a>0$, there exists $C_0(n,\a)>0$ such that the following holds: Suppose $(M^n,g(t))$ is a complete solution to the Ricci flow on $M\times [0,T]$ with $g(0)=g_0$ and satisfies 
\[
|\Rm(g(t))|\leq \a t^{-1} \text{  for some  } \a>0.
\]
Then for $0\le 2s\le t\le T$ and $x,y\in B_{g_0}(x_0,r)$,
\begin{equation}\label{G est eqn}
G(x,t;y,s)\le \frac{C_0}{\mathrm{Vol}_{g(t)}\left(B_{g(t)}(x,\sqrt{t})\right)}\exp\left(-\frac{d_{g(s)}^2(x,y)}{C_0t}\right).    
\end{equation}
where $G$ denotes the dirichlet heat kernel on $\Omega=B_{g_0}(x_0,2r)$. The same estimate also holds for heat kernel on $M$.
\end{prop}
\begin{rem} It is not difficult to see that the completeness condition of $g(t)$ can be replaced by the compactness of a suitable geodesic ball centred at $x_0$ and suitable curvature condition initially, for instances see \cite{LeeTam2022}. 
\end{rem}

Before proving Proposition \ref{prop: heat Kernel}, let's recall the following estimate in \cite{LeeTam2022} which is originated in \cite{ChauTamYu2011}.

\begin{lma}\cite{LeeTam2022}\label{LT4.1}
Let $(M^n, g_0)$ be a Riemannian manifold and $p\in M$. Suppose that $g(t)$ is a complete solution to the Ricci flow on $M\times [0,1]$ with $g(0)=g_0$ such that $|\Rm(x,t)|\le A$ on $M\times [0,1]$ for some $A>0$. If $\Omega$ is an open subset in $M$ with smooth boundary such that $\Omega\Subset B_0(p,r)$ and $G_{\Omega}(x,t;y,s)$ is the dirichlet heat kernel with respect to the backward heat flow on $\Omega\times \Omega\times [0,1]$. Then there exists a positive constant $C_1(n,A)>0$ such that for all $0\le s<t\le 1$, $0\le \tau\le 1$, $x,y\in \Omega$,
\[
G_{\Omega}(x,t;y,s)\le \frac{C_1}{\mathrm{Vol}_{g(\tau)}\left(B_{g(\tau)}(x,\sqrt{t-s})\right)}\times \exp\left(-\frac{d_{g(\tau)}^2(x,y)}{C_1(t-s)}\right);
\]
\[
G_{\Omega}(x,t;y,s)\le \frac{C_1}{\mathrm{Vol}_{g(\tau)}\left(B_{g(\tau)}(y,\sqrt{t-s})\right)}\times \exp\left(-\frac{d_{g(\tau)}^2(x,y)}{C_1(t-s)}\right).
\]
\end{lma}
\begin{proof} Unless specified, the constants in this proof depend only on $n$ and $A$. Their exact value may change from line to line.
The estimate was stated in slightly different way in \cite[Lemma 4.1]{LeeTam2022}. We sketch the idea of how to derive the estimate in the lemma from \cite{LeeTam2022} for the sake of completeness. Let $x,y\in \Omega$ and $d=d_{g_0}(x,y)$. For $d\le \sqrt{t-s}$, by \cite[Lemma 4.1]{LeeTam2022},
\begin{equation}
    \begin{split}
G_{\Omega}(x,t;y,s)&\le \min\left\{\frac{C}{\V_{g_0} \left(B_{g_0} (x,\sqrt{t-s})\right) },\frac{C}{\V_{g_0} \left(B_{g_0} (y,\sqrt{t-s})\right) }\right\}\\
&\le\frac{Ce^{C^{-1}}}{\V_{g_0} \left(B_{g_0} (x,\sqrt{t-s})\right) }\times \exp\left(-\frac{d_{g_0}^2(x,y)}{C(t-s)}\right).
    \end{split}
\end{equation}

For $d> \sqrt{t-s}$, by the volume comparison theorem 
\begin{equation}
    \begin{split}
&\quad \exp\left(-\frac{d_{g_0}^2(x,y)}{2C(t-s)}\right)\frac{ \V_{g_0}^\frac12\left( B_{g_0}(x,\sqrt{t-s}\right))}{\V_{g_0}^\frac12\left( B_{g_0}(y,\sqrt{t-s}\right))}\\
&\le \exp\left(-\frac{d_{g_0}^2(x,y)}{2C(t-s)}\right)\frac{ \V_{g_0}^\frac12\left( B_{g_0}(y,d+\sqrt{t-s}\right))}{\V_{g_0}^\frac12\left( B_{g_0}(y,\sqrt{t-s}\right))}\leq C
    \end{split}
\end{equation}

Hence by \cite[Lemma 4.1]{LeeTam2022}, we have 
\begin{equation}
    \begin{split}
G_{\Omega}(x,t;y,s)&\le \frac{C}{\V_{g_0}^\frac12\left( B_{g_0}(x,\sqrt{t-s}\right))\V_{g_0}^\frac12\left( B_{g_0}(y,\sqrt{t-s}\right))}\times \exp\left(-\frac{d_{g_0}^2(x,y)}{C(t-s)}\right)\\
&\le \frac{C}{\V_{g_0}\left( B_{g_0}(x,\sqrt{t-s}\right))}\times \exp\left(-\frac{d_{g_0}^2(x,y)}{2C(t-s)}\right).
  \end{split}
\end{equation}
This completes the proof of the lemma for $\tau=0$. For general $\tau\in [0,1]$, $|\Rm|\le A$, we have
\[
e^{-(n-1)A}d_{g(\tau)}\le d_{g_0}\le e^{(n-1)A}d_{g(\tau)}\;\;\text{  on } M.
\]
Hence by the volume comparison and $|\Rm|\le A$,
\begin{equation}
    \begin{split}
\V_{g_0}\left( B_{g_0}(x,\sqrt{t-s}\right))&\ge c(n,A)\V_{g(\tau)}\left( B_{g(\tau)}(x,e^{-(n-1)A}\sqrt{t-s}\right))\\
&\ge c_1(n,A)\V_{g(\tau)}\left( B_{g(\tau)}(x,\sqrt{t-s}\right)).
 \end{split}
\end{equation}
The estimate on $G_{\Omega}$ for general $\tau$ then follows from the estimate when $\tau=0$. By switching the role of $x$ and $y$, we get the second estimate.
\end{proof}

Now we are in position to prove Proposition~\ref{prop: heat Kernel} using the idea in \cite{BamlerCabezasWilking2019}. 
\begin{proof}[Proof of Proposition \ref{prop: heat Kernel}]
The proof is almost identical to \cite[Proposition 3.1]{BamlerCabezasWilking2019}. Since the relaxation of non-collapsing assumption is important, we include the proof for readers' convenience. We only work on the dirichlet heat kernel. The global heat kernel follows from a similar argument. 
    By parabolic scaling and time shift, W.L.O.G., we may assume that $t=1$, $s=0$.

It suffices to show the following
\begin{equation}\label{gl after scal}
    G_\Omega(x,1;y,0)\le \frac{C}{\V_{g(1)}\left(B_{g(1)}(x,1)\right)}\exp\left(-\frac{d_{g(0)}^2(x,y)}{C}\right).
\end{equation}

We write $(0,1]=\cup_{k=0}^{\infty}[t_{k+1}, t_k]$, where $t_k:=16^{-k}$. By Lemma \ref{LT4.1}, there is a constant $C(n,\alpha)$ such that for all integer $k\ge 0$, $x,y \in \Omega$, $\tau\in [t_{k+1}, t_k]$,
\begin{equation}\label{yG est}
    \begin{split}
        G_{\Omega}(x,t_k;y,t_{k+1})\le \frac{C}{\V_{g(\tau)}\left(B_{g(\tau)}(x,\sqrt{t_k-t_{k+1}})\right)}\times \exp\left(-\frac{d_{g(\tau)}^2(x,y)}{C(t_k-t_{k+1})}\right);\\
         G_{\Omega}(x,t_k;y,t_{k+1})\le \frac{C}{\V_{g(\tau)}\left(B_{g(\tau)}(y,\sqrt{t_k-t_{k+1}})\right)}\times \exp\left(-\frac{d_{g(\tau)}^2(x,y)}{C(t_k-t_{k+1})}\right).
    \end{split}
\end{equation}

In particular, by taking $k=0$ and $\tau=t_1$, for any $y\in \Omega$,
\begin{equation}\label{a1 est}
G_{\Omega}(x,1;y,t_{1})\le \frac{C}{\V_{g(t_1)}\left(B_{g(t_1)}(x,\sqrt{1-t_{1}})\right)}\times \exp\left(-\frac{d_{g(t_1)}^2(x,y)}{C(1-t_1)}\right).
\end{equation}
Applying the maximum principle to $G_{\Omega}(x,1;\cdot,\cdot)$ on $\Omega\times [0,t_1]$ with dirichlet boundary condition, we have 
\begin{equation}\label{G rough bdd}
G_{\Omega}(x,1;\cdot,\cdot)\le \frac{C}{\V_{g(t_1)}\left(B_{g(t_1)}(x,\sqrt{1-t_{1}})\right)}\;\; \text{  on } \Omega\times [0,t_1].
\end{equation}
Let $d$ be a large constant with lower bound depending on $n$ and $\alpha$ and to be determined later. For positive integer $k$, let $r_k:=4d(1-2^{-k})$ and 
\[
a_k:=
\begin{cases}
\displaystyle\sup_{\Omega\setminus B_{g_0}(x,r_k)}G_{\Omega}(x,1;\cdot,t_k)\quad\text{  if }&\Omega\setminus B_{g_0}(x,r_k)\neq \phi;\\
0 &\text{  otherwise.}
\end{cases}
\]
By the continuity of $G_{\Omega}$, 
\begin{equation}\label{ak vs G}
\lim_{k\to\infty}a_k\ge\displaystyle\sup_{\Omega\setminus B_{g_0}(x,4d)}G_{\Omega}(x,1;\cdot,0).
\end{equation}
We claim that there are positive constants $C(n,\alpha)$ and $\underline{d}(n,\alpha)$ such that for all $d\ge\underline{d}$ and positive integer $k$, 
\[
a_{k+1}\le  \frac{C}{\V_{g(t_1)}\left(B_{g(t_1)}(x,\sqrt{1-t_{1}})\right)}\exp\left( -\frac{d^2}{C}\right).
\]
Let us assume the above claim and prove \eqref{gl after scal}. If $d_{g_0}(x,y)\le 4\underline{d}$, then by \eqref{G rough bdd} and the volume comparison,

\begin{equation}
\begin{split}
G_{\Omega}(x,1;y,0)&\le \frac{C}{\V_{g(t_1)}\left(B_{g(t_1)}(x,\sqrt{1-t_{1}})\right)}\\
&\le \frac{C}{\V_{g(1)}\left(B_{g(1)}(x,1)\right)}\exp\left(-\frac{d_{g_0}^2(x,y)}{C}\right).
    \end{split}
\end{equation}
Then \eqref{gl after scal} holds in this case. Suppose $d_{g_0}(x,y)> 4\underline{d}$. taking $d=d_{g_0}(x,y)/4$, we see from \eqref{ak vs G} and the claim that
\begin{equation}
\begin{split}
G_{\Omega}(x,1;y,0)&\le  \frac{C}{\V_{g(t_1)}\left(B_{g(t_1)}(x,\sqrt{1-t_1})\right)}\exp\left(-\frac{d_{g_0}^2(x,y)}{16C}\right)\\
&\le\frac{C}{\V_{g(1)}\left(B_{g(1)}(x,1)\right)}\exp\left(-\frac{d_{g_0}^2(x,y)}{16C}\right)\\
   \end{split}
\end{equation}
for some $C_1(n,\alpha)>0$. It remains to justify the claim. By the semi-group property, for any $y\in \Omega\setminus B_{g_0}(x, r_{k+1})$,
\[
G_{\Omega}(x, 1; y, t_{k+1})=\int_{\Omega}G_{\Omega}(x, 1; z, t_k)G_{\Omega}(z, t_k; y, t_{k+1})\,d \mathrm{vol}_{g(t_k)}(z).
\]
We estimate the integral on the right hand side by splitting it into the integrals over $B_k\cap \Omega$ and over $\Omega\setminus B_k$, where $B_k:=B_{g(t_k)}(y, d/2^k)$. The later integral is understood to be $0$ if $\Omega\setminus B_k$ happens to be empty. By Hamilton-Perelman distance distortion estimate \cite{Perelman2002} (c.f. \cite{SimonTopping2017}), if $z\in \Omega\cap B_{g_0}(x, r_k)$, then
\begin{equation}
\begin{split}
    d_{g(t_k)}(y,z)&\ge d_{g_0}(y,z)-c_n\sqrt{\a}\int_0^{t_k}t^{-1/2}\,dt\\
    &\ge d_{g_0}(y,x)-d_{g_0}(x,z)-2c_n\sqrt{\a t_k} \geq  2^{-k}d
   \end{split}
\end{equation}
provided $d\ge c_n\sqrt{\a}$. Hence 
$\Omega\cap B_k\subset \Omega\setminus B_{g_0}(x,r_k)$ and 
\begin{equation}
\begin{split}
&\int_{B_k\cap\Omega}G_{\Omega}(x, 1; z, t_k)G_{\Omega}(z, t_k; y, t_{k+1})\,d \mathrm{vol}_{g(t_k)}(z)\\
&\le \int_{\Omega\setminus B_0(x,r_k)}G_{\Omega}(x, 1; z, t_k)G_{\Omega}(z, t_k; y, t_{k+1})\,d \mathrm{vol}_{g(t_k)}(z)\leq  a_k.
   \end{split}
\end{equation}
By \eqref{yG est}, \eqref{G rough bdd} and the volume comparison, 
\begin{equation}
\begin{split}
&\quad \int_{\Omega\setminus B_k}G_{\Omega}(x, 1; z, t_k)G_{\Omega}(z, t_k; y, t_{k+1})\,d \mathrm{vol}_{g(t_k)}(z)\\
&\le\frac{C}{\V_{g(t_1)}\left(B_{g(t_1)}(x,\sqrt{1-t_{1}})\right)} \int_{\Omega\setminus B_k}G_{\Omega}(z, t_k; y, t_{k+1})\,d \mathrm{vol}_{g(t_k)}(z)\\
&\le \frac{2C'}{\V_{g(t_1)}\left(B_{g(t_1)}(x,\sqrt{1-t_{1}})\right)}\exp\left(-\frac{4^{-k}d^2}{2C'(t_k-t_{k+1})}\right).
\end{split}
\end{equation}
Combining the two integrals, we have 
\begin{equation}
\begin{split}
a_{k+1}&\le a_k+\frac{2C'}{\V_{g(t_1)}\left(B_{g(t_1)}(x,\sqrt{1-t_{1}})\right)}\exp\left(-\frac{4^{-k}d^2}{2C'16^{1-k}}\right)\\
&\le a_1+\frac{C''}{\V_{g(t_1)}\left(B_{g(t_1)}(x,\sqrt{1-t_{1}})\right)}\sum_{i=1}^k\exp\left(-\frac{4^i d^2}{C''}\right)\\
&\le \frac{C'''}{\V_{g(t_1)}\left(B_{g(t_1)}(x,\sqrt{1-t_{1}})\right)}\Big[\exp\left(-\frac{d_{g(t_1)}^2(x,y)}{C(1-t_1)}\right)\\
&\quad +\exp\left(-\frac{d^2}{C''}\right)\sum_{i=1}^k\exp\left(-\frac{(4^i-1) d^2}{C''}\right)\Big],
\end{split}
\end{equation}
where $y\in \overline{\Omega}\setminus B_0(x,2d)$. We also used \eqref{a1 est} to bound $a_1$ in the last inequality. Again by Hamilton Perelman distance distortion \cite{Perelman2002} (c.f. \cite{SimonTopping2017}), for $d\gg \underline{d}(n,\a)$, we have
\[
d_{g(t_1)}(x,y)\ge d_{g_0}(x,y)-c_n\sqrt{\a}\int_0^{t_1}t^{-1/2}\,dt\ge 2d-2c_n\sqrt{\a}\ge d.
\]
\[
a_{k+1}\le \frac{C}{\V_{g(t_1)}\left(B_{g(t_1)}(x,\sqrt{1-t_{1}})\right)}\exp\left(-\frac{d^2}{C}\right)
\]
as required.
\end{proof}

\subsection{Almost monotonicity of local entropy}
In view of Proposition~\ref{prop: heat Kernel}, it is important to compare the volume of the evolving ball with the initial one.  We will make use of the entropy. We recall the concept of local entropy introduced by Wang \cite{Wang2018}. Let $\Omega$ be a connected domain with possibly empty boundary in $M$, denote $$ D_g(\Omega):=\left\{u: u\in W^{1,2}_0(\Omega), u\geq 0 \text{  and  } \|u\|_{L^2(\Omega)}=1 \right\}$$ and consider the following quantities
\begin{equation}
\left\{
\begin{array}{ll}
&\displaystyle  W(\Omega, g, u, \tau):=\int_{\Omega}\left[\tau(\mathcal{R}u^2+4|\nabla u|^2)-2u^2\log u \right]\;   d\mathrm{vol}_g -\frac{n}{2}\log(4\pi\tau)-n;
\\\\
&\displaystyle  \mu(\Omega, g, \tau):= \inf_{u\in D_g(\Omega)}W(\Omega, g, u, \tau);
\\\\
&\displaystyle  \nu(\Omega, g, \tau):= \inf_{s\in (0,\tau]}\mu(\Omega,g,s)
\end{array}
\right.
\end{equation}
where $\mathcal{R}$ denotes the scalar curvature of $(M,g)$. 

The following Lemmas provide us the relation between entropy and the volume ratio under Ricci lower bound.
\begin{lma}\label{lma:Ric-to-Volratio} There exists $C_n>0$ such that the following holds.  Suppose $B_{g_0}(x_0,2r)\subset M$ with $\partial B_{g_0}(x_0,2r)\neq \emptyset$ and $\Ric(g_0)\geq -r^{-2}$ on $B_{g_0}(x_0,2r)$, then 
\begin{equation}\label{lma:Ric-ineq}
\log \frac{\mathrm{Vol}_{g_0}(B_{g_0}(x_0,r))}{r^n}\leq C_n+ \nu\left(B_{g_0}(x_0,r), g_0,r^2\right).
\end{equation}
In particular, if $\Ric(g_0)\geq 0$ on $M$, then \eqref{lma:Ric-ineq} holds for all $r>0$.
\end{lma}
\begin{proof}
This follows from \cite[Corollary 2.2]{TianZhang2021} and \cite[Theorem 3.6]{Wang2018}.
\end{proof}

The next result is proved by Wang stating that the entropy lower bound also gives rise to a lower bound of volume ratio under a bound on scalar curvature.
\begin{lma}\label{lma:entrop-back-Vol}Suppose $B_{g_0}(x_0,2r)\subset M$ is a geodesic ball with $\partial B_{g_0}(x_0,2r)\neq \emptyset$ and $\mathcal{R}(g_0)\leq \Lambda$ on $B_{g_0}(x_0,2r)$, then we have 
$$\log \frac{\mathrm{Vol}_{g_0}(B_{g_0}(x_0,r))}{r^n} \geq \nu\left(B_{g_0}(x_0,r), g_0,r^2\right)-C_n-\Lambda r^2.$$
\end{lma}
\begin{proof}
This follows directly from \cite[Theorem 3.3]{Wang2018}.
\end{proof}

When $g(t)$ is a complete Ricci flow on $M$ with bounded curvature, it is proved by Wang \cite[Theorem 5.3]{Wang2018} building on the work of Chau-Tam-Yu \cite{ChauTamYu2011} (see also the work of Perelman \cite{Perelman2002}) that the entropy is monotone.  Together with Lemma~\ref{lma:Ric-to-Volratio} and Lemma~\ref{lma:entrop-back-Vol}, it morally says that the volume ratio is (almost) monotonic increasing.  We will need the following almost monotonicity of entropy of Wang \cite[Theorem 5.4]{Wang2018}.
\begin{thm}\label{thm:Wang-entropy} Let $(M^n,g(t)),t\in [0,T]$ be a complete Ricci flow with bounded curvature. Suppose $x_0\in M$ and $\a\geq 10^3n$ is a constant such that for all $x\in B_{g(t)}(x_0,\sqrt{t}),t\in (0,T]$, 
$$\Ric(x,t)\leq (n-1)\a t^{-1}.$$
Then for any $\tau\in (0,\a^2 T]$, we have
$$\nu\left(B_{g(T)}(x_0,8\a \sqrt{T}),g(T),\tau \right) \geq-\a^{-2}+ \nu\left(B_{g_0}(x_0,20\a \sqrt{T}),g(0),\tau+T \right).$$
\end{thm}

\section{a-priori estimates of bounded curvature Ricci flow}
In this section, we will consider \textit{complete bounded curvature} Ricci flow and derive local curvature estimates. Our goal is to prove the following pseudo-locality type Theorem. This in particular proves the compact case in Theorem~\ref{Thm:RF-existence}.

\begin{thm}\label{thm:pseudo-locality} 
    For any $n\geq 3$, there exist $\e_0(n),\a_n, S_n,L_n>0$ such that the following holds: 
   Suppose $(M,h_0)$ 
   is a complete manifold with bounded curvature. If   for some $r>0$, the initial metric $h_0$ satisfies 
    \begin{enumerate}
        \item[(a)] for some $0<\Lambda_0<1$, either 
       \begin{enumerate}
           \item[(i)] $\mathrm{Rm}(h_0)+\Lambda_0 r^{-2}h_0\owedge h_0/2\in \mathrm{C}_{\mathrm{PIC1}}$ if $n\geq 4$; or
           \item[(ii)] $K(h_0)+\Lambda_0 r^{-2}\geq 0$ if $n=3$,
       \end{enumerate} 
       \item[(b)] for all $x\in M$, 
       $$ \int^{r}_0 s\left(\fint_{B_{g_0}(x,r)}|\Rm(h_0)|\, d\mathrm{vol}_{h_0} \right)\, ds<\e_0.$$
    \end{enumerate}
    Then the complete bounded curvature Ricci flow $h(t)$ on $M$ 
    starting with $h(0)=h_0$ exists up to $S_n \cdot (r^2\wedge \mathrm{diam}(M,h_0)^2)$ and satisfies 
    \begin{enumerate}
        \item[(I)] $\sup_N |\Rm(h(t))|\leq \a_n t^{-1}$;
        \item [(II)] $\mathrm{Rm}(h(t))+L_n\Lambda_0  r^{-2}h(t)\owedge h(t)/2 \in \mathrm{C}_{\mathrm{PIC1}}$ if $n\geq 4$;
        \item [(III)] $K(h(t))+L_n \Lambda_0 r^{-2}\geq 0$ if $n=3$.
    \end{enumerate}
    If $M$ is complete non-compact, $\mathrm{diam}(M,h_0)$ is understood to be $+\infty$. 
\end{thm}

For notation convenience, we denote
\begin{equation}
\left\{
\begin{array}{ll}
&k_{h_0}(x,r)=\displaystyle r^2\fint_{B_{h_0}(x,r)}|\Rm(h_0)|\, d\mathrm{vol}_{h_0};\\
&f_{h_0}(x,r)=\displaystyle\int^r_0 s^{-1}k_{h_0}(x,s)ds.
\end{array}
\right.
\end{equation}
The notation is  scaling invariant in the sense that $k_{\lambda^2 h_0}(\lambda r)=k_{h_0}(r)$ and $f_{\lambda^2h_0}(\lambda r)=f_{h_0}(r)$ for $\lambda>0$. We will omit the index $h_0$ if the content is clear.

We first observe that the bound on $f(x,r)$ will imply a bound on $k(x,r/2)$ by a simple comparison argument under assumption on Ricci lower bound. 
\begin{lma}\label{lma:compare-k-f}
For any $n\geq 3$, there exists $C_n>0$ such that the following holds: Suppose $\Ric(h_0)\geq -(n-1)$ on $M$, then for all $r\in (0,1]$ and $x\in M$, $$k(x,r)\leq C_n (f(x,2r)-f(x,r)).$$ 
\end{lma}
\begin{proof}
   By volume comparison  for all $s\in [r,2r]$ and $0<r\leq 1$,
      \begin{equation}
        \begin{split}
            k(x_0,r) &=r^2 \fint_{B_{h_0}(x_0,r)} |\Rm(h_0)| \,d\mathrm{vol}_{h_0} \\
            &\leq \frac{r^2}{s^2}\frac{{\mathrm{Vol}_{h_0} (B_{h_0}(x_0,s))}}{\mathrm{Vol}_{h_0} (B_{h_0}(x_0,r))} s^2\fint_{B_{h_0}(x_0,s)} |\Rm(h_0)| \,d\mathrm{vol}_{h_0} \\
            &\leq  C_n \cdot k(x_0,s).
        \end{split}
    \end{equation}
Hence, 
\begin{equation}
    \begin{split}
        k(x_0,r) &\leq C_n\int^{2r}_r s^{-1}k(x_0,s)  ds \\
        &=  C_n \left( f(x_0,2r)-f(x_0,r)\right). 
    \end{split}
\end{equation}
\end{proof}

The key observation is a curvature inequality, which says that $\mathcal{R}\, g-2\Ric\ge 0$ under nonnegative sectional curvature condition (see, for instances, \cite[Proposition 5.4]{MunteanuSungWang2019} and \cite[Lemma 4.4]{Lai2022}).
We observe that a similar inequality still holds under $\mathrm{PIC}_1$ condition, this eventually leads to an improvement on the evolution inequality of scalar curvature under almost $\mathrm{PIC}_1$ condition.
\begin{lma}\label{lma:R-improve}
Suppose $\mathrm{R}\in \mathrm{C}$, i.e. the cone of  curvature type operator, such that 
\begin{enumerate}
    \item[(a)] $\mathrm{R}\in \mathrm{C}_{\mathrm{PIC1}}$ if $n\geq 4$;
    \item[(b)] $\mathrm{K}(\mathrm{R})\geq 0$ if $n\geq 3$,
\end{enumerate}
then $$\Ric(\mathrm{R}) \leq \frac12\mathrm{scal}(\mathrm{R})\cdot g.$$
\end{lma} 
\begin{proof}
Suppose $\lambda_i$ is the eigenvalues of $\Ric(\mathrm{R})$ with respect to $g$ so that $\lambda_n\geq \lambda_{n-1}\geq ...\geq \lambda_1$. It suffices to show that $\sum_{i=1}^n \lambda_i \geq 2\lambda_n$. 
\begin{equation}
    \begin{split}
        \sum_{i=1}^n \lambda_i -2\lambda_n&=\sum_{i=1}^{n-1} \lambda_i -\lambda_n\\
        &= \sum_{i=1}^{n-1}\sum_{j=1}^n R_{ijji}- \sum_{i=1}^{n-1} R_{inni}\\
        &=\sum_{i,j=1}^{n-1}R_{ijji}.
    \end{split}
\end{equation}
Since $\mathrm{R}\in \mathrm{C}_{\mathrm{PIC1}}$, $R_{ijji}+R_{ikki}\geq 0$ for all $i,j,k$ distinct. Without loss of generality, assume $i=1$ 
\begin{equation}
\begin{split}
   2 \sum_{j=2}^{n-1}R_{ijji}&= \left( \sum_{j=2}^{n-2} (R_{1jj1}+R_{1 (j+1)(j+1)1})\right)+R_{1(n-1)(n-1)1}+R_{1221} \\
&\geq 0.
\end{split}
\end{equation}
Replacing $i=1$ by $i\in \{2,...,n-1\}$, we see that the sum is non-negative. Alternatively, $\sum_{i,j=1}^{n-1}R_{ijji}$ is the scalar curvature of $\mathrm{R}$ restricted on the orthogonal sub-space of $e_{n}$ which is non-negative using  $\mathrm{R}\in \mathrm{C}_{\mathrm{PIC1}}$. The case of $n=3$ follows directly from non-negativity of sectional curvature. 
\end{proof}

Next, we will make use of Wang's almost monotonicity of entropy to improve the heat kernel estimate in Proposition~\ref{prop: heat Kernel}. This is motivated by the relative volume comparison proved in \cite{TianZhang2021}.
\begin{lma}\label{lma:improved-heat}
For any $n,\a>0$, there exists $C_1(n,\a)>0$ such that the following holds: Suppose $(M^n,h(t))$ 
is a complete solution to the Ricci flow on $M\times [0,T]$ 
with $h(0)=h_0$ and satisfies 
\begin{equation}
    \left\{
    \begin{array}{ll}
     & |\Rm(h(t))|\leq \a t^{-1} \text{  for some  } \a>10^3n;\\
      &  \Ric(h_0) \geq -(n-1).
    \end{array}
    \right.
\end{equation}
Then for $0< t\le \min\{T,(20\a)^{-2}, (10^3(n-1)^2 \a^2)^{-1} \mathrm{diam}(M,h_0)^2\}$ and $x,y\in B_{h_0}(x_0,r)$,
\begin{equation}\label{new-G est eqn}
 G(x,t;y,0)\le \frac{C_1}{\mathrm{Vol}_{h_0}\left(B_{h_0}(x,\sqrt{t})\right)}\exp\left(-\frac{d_{h_0}^2(x,y)}{C_1t}\right),
\end{equation}
where $G$ denotes the dirichlet heat kernel on $\Omega=B_{h_0}(x_0,2r)$. The same also holds for heat kernel on $M$. 
If $M$ is non-compact, $\mathrm{diam}(M,h_0)$ is understood to be $+\infty$. 
\end{lma}
\begin{proof}
We first focus on the case when $M$ is complete non-compact. 
We apply Lemma~\ref{lma:entrop-back-Vol}, monotonicity of entropy \cite[Proposition 2.1]{Wang2018} and Theorem~\ref{thm:Wang-entropy} to $h(t)$ so that for each $t\in (0,T\wedge (20\a)^{-2}]$,
    \begin{equation}\label{vol-to-ent}
    \begin{split}
        \log \frac{ \mathrm{Vol}_{h(t)}\left(B_{h(t)}(x,\sqrt{t})\right)}{t^{n/2}}
       &\geq \nu\left(B_{h(t)}(x_0,\sqrt{t}), h(t),t\right)-C_n-\a\\
       &\geq \nu\left(B_{h(t)}(x_0,8\a \sqrt{t}), h(t),t\right)-C_n-\a\\
       &\geq \nu\left(B_{h_0}(x_0,20\a \sqrt{t}), h_0,2t\right)-C_n'-\a.
    \end{split}
    \end{equation}

Furthermore,  Lemma~\ref{lma:Ric-to-Volratio} implies
\begin{equation}\label{ent-to-vol}
    \begin{split}
        \nu\left(B_{h_0}(x_0,20\a \sqrt{t}), h_0,2t\right)
        &\geq  \nu\left(B_{h_0}(x_0,20\a \sqrt{t}), h_0,20^2\a^2t\right)\\
        &\geq \log \frac{ \mathrm{Vol}_{h_0}\left(B_{h_0}(x,20\a \sqrt{t})\right)}{(20\a \sqrt{t})^{n}}-C_n.
    \end{split}
\end{equation}

Combines \eqref{vol-to-ent} and \eqref{ent-to-vol}, we conclude that 
\begin{equation}
    \begin{split}
          \mathrm{Vol}_{h(t)}\left(B_{h(t)}(x,\sqrt{t})\right)&\geq e^{-C(n,\a)}\cdot  \frac{ \mathrm{Vol}_{h_0}\left(B_{h_0}(x,20\a \sqrt{t})\right)}{(20\a )^{n}}\\
          &\geq e^{-C'(n,\a)} \mathrm{Vol}_{h_0}\left(B_{h_0}(x, \sqrt{t})\right).
    \end{split}
\end{equation}

Result follows by combining this with Proposition~\ref{prop: heat Kernel}.
\vskip0.3cm 

When $M$ is compact, we need to pay more attention to the boundary of ball where we use the entropy. Using \cite[Corollary 3.3]{SimonTopping2016} and curvature assumption, we have 
\begin{equation}
    \mathrm{diam}(M,h_0)\leq \mathrm{diam}(M,h(t))+8(n-1)\sqrt{\a t}.
\end{equation}
Since $t\leq 10^{-3}(n-1)^{-2} \a^{-2}\mathrm{diam}(M,h_0)^2$, it follows that the boundary of both $ B_{h(t)}(x,\sqrt{t})$ and $B_{h_0}(x,20\a \sqrt{t})$ are non-empty so that both Lemma~\ref{lma:entrop-back-Vol} and Lemma~\ref{lma:Ric-to-Volratio} are applicable. 
\end{proof}

Now we are ready to prove Theorem~\ref{thm:pseudo-locality}. This is inspired by \cite[Theorem 1]{BamlerCabezasWilking2019}. 
\begin{proof}[Proof of Theorem~\ref{thm:pseudo-locality}] 
We focus on the case of $n\geq 4$ while the case of $n=3$ can be done by replacing the cone $\mathrm{C}_{\mathrm{PIC1}}$ with the cone of non-negative sectional curvature.

We will specify the choice of $\e_0(n)$. By scaling, it suffices to show that the conclusion holds for some $r_0$. We let $r_0$ to be a large constant $L(n)$ in which $L^{-2}<1$.  We fix $\a(n)=  10^3n\b_n$ where $\b_n>1$ is a dimensional constant so that $|\mathrm{R}|\leq \b_n \mathrm{scal}(\mathrm{R})$ for all $\mathrm{R}\in \mathrm{C}_{\mathrm{PIC1}}$.  In what follows, we will use $C_i$ to denote constants depending only on $n$.

We first note that by applying Lemma~\ref{lma:compare-k-f} to $L^{-2}h_0$, we have in addition for all $x\in M$,
\begin{equation}
    k\left(x,\frac12 r_0\right)\leq C_n \e_0.
\end{equation}
Since $f(x,r)$ is non-decreasing in $r>0$, we might assume in addition \begin{equation}k(x,r_0)\leq C_n\e_0\end{equation} by replacing $r_0$ with $\frac12 r_0$.

By the existence theory of Shi \cite{Shi1989}, it admits a short-time solution $h(t)$ to the Ricci flow. We consider the maximal existence time interval $[0,t_1)$ the Ricci flow $h(t)$ exists and satisfies
\begin{enumerate}
    \item[(i)] $|\Rm(h(t))|<\a t^{-1}$;
        \item[(ii)]   $\mathrm{Rm}(h(t))+ \tfrac{1}{2}\Lambda_0 h(t)\owedge h(t)\in \mathrm{int}\left(\mathrm{C}_{\mathrm{PIC1}}\right)$.
\end{enumerate}
We claim that if $\e_0$ is sufficiently small, then $t_1$ is uniformly bounded from below. 

We first consider the condition (ii). By \cite{BamlerCabezasWilking2019},  the function
$$\phi(x,t)=\displaystyle\inf\left\{ s>0: \mathrm{Rm}(h(x,t))+\tfrac{1}{2} s\, h(x,t)\owedge h(x,t)\in \mathrm{C}_{\mathrm{PIC1}}\right\}$$
satisfies $\rheat \phi \leq \mathcal{R}\phi+c_n\phi^2$ in the sense of barrier and hence in the distributional sense (see \cite[Appendix]{MantegazzaMascellaniUraltsev2014}). By assumption, $\phi(x,0)\leq \Lambda_0 L^{-2}$, $\phi(x,t)\leq 1$ and $\rheat \phi \leq \mathcal{R}\phi +c_n\phi$ so that maximum principle implies 
\begin{equation}
    \begin{split}
        e^{-c_nt}\phi(x,t)&\leq \int_M G(x,t;y,0) \phi(y,0) \,d\mathrm{vol}_{h_0}(y).
    \end{split}
\end{equation}
Using Lemma~\ref{lma:improved-heat}, Stokes' Theorem and volume comparison,
\begin{equation}
    \begin{split}
        e^{-c_nt}\phi(x,t)&\leq \int_0^\infty \frac{\V_{h_0}\left(B_{h_0}(x,r)\right)}{\V_{h_0}\left(B_{h_0}(x,\sqrt{t})\right)}\exp\left(-\frac{r^2}{C_1t}\right)\frac{2\Lambda_0 r}{L^2 t} \, dr\\
        &= \left(\int_0^{\sqrt{t}}+\int^\infty_{\sqrt{t}}\right) \frac{\V_{h_0}\left(B_{h_0}(x,r)\right)}{\V_{h_0}\left(B_{h_0}(x,\sqrt{t})\right)}\exp\left(-\frac{r^2}{C_1t}\right)\frac{2\Lambda_0 r}{L^2 t} \, dr\\
        &\leq \int_0^{\sqrt{t}} \exp\left(-\frac{r^2}{C_1t}\right)\frac{2\Lambda_0 r}{L^2 t} \, dr +\int_{\sqrt{t}}^\infty  \exp\left(-\frac{r^2}{C_1t}+C_n\frac{r}{\sqrt{t}}\right)\frac{2\Lambda_0 r}{L^2 t} \, dr\\
        &\leq \frac{C_2\Lambda_0 }{L^2}\int_0^\infty  \exp\left(-\frac{r^2}{C_2}\right)dr=C_3\Lambda_0 L^{-2}.
    \end{split}
\end{equation}

Now we fix $L(n)= 2\sqrt{C_3}$ so that on $M\times [0,t_1]$, 
$ \phi(x,t)\leq \frac14\Lambda_0  e^{c_nt}$ and therefore, 
\begin{equation}\label{PIC1-control}
    \mathrm{Rm}(h(t))+\frac18 \exp(c_nt)\Lambda_0  h(t)\owedge h(t)\in \mathrm{int}\left(\mathrm{C}_{\mathrm{PIC1}}\right).
\end{equation}

\bigskip

Now we proceed to estimate (i). On $[0,t_1]$, define the twisted curvature operator 
$\widetilde{\mathrm{Rm}}=\mathrm{Rm}(h(t))+h(t)\owedge h(t)\in \mathrm{C}_{\mathrm{PIC1}}$ 
and its corresponding scalar $\tilde {\mathcal{R}}$ and Ricci curvature $\widetilde {\Ric}$ so that $\tilde {\mathcal{R}}\geq  n(n-1)$.

By the Ricci flow equation and Lemma~\ref{lma:R-improve}, it satisfies
\begin{equation}
    \begin{split}
        \rheat \tilde{\mathcal{R}}&=2|\Ric|^2\\
        &=2|\widetilde{\Ric}-2(n-1) h(t)|^2\\
        &\leq \tilde {\mathcal{R}}^2 -8(n-1) \tilde {\mathcal{R}} +8 n(n-1)^2  \\
        &\leq  \mathcal{R}\cdot \tilde {\mathcal{R}} +2n(n-1) \tilde {\mathcal{R}} .
    \end{split}
\end{equation}

Hence, the function $\varphi=e^{-2n(n-1) t}\tilde {\mathcal{R}} $ satisfies 
\begin{equation}
\begin{split}
    \rheat \varphi&\leq \mathcal{R} \varphi.
\end{split}
\end{equation}

Fix an arbitrary $x_0\in M$ and define 
$$u(x,t)=\int_\Omega G_\Omega(x,t;y,0)\varphi(y,0)\,d\mathrm{vol}_{h_0}(y)$$
where $\Omega=B_{h_0}(x_0,\frac12 r_0)$ and $G_\Omega$ is the dirichlet heat kernel on $\Omega$ if $\Omega$ is a proper subset of $M$. If $\Omega=M$, then we take $G$ to be the global heat kernel. In this way, $v=\varphi-u$ satisfies $\rheat v\leq \mathcal{R} v$ and $v=0$ in the interior of $\Omega$. Then the average twisted scalar curvature satisfies
\begin{equation}
    \begin{split}
        0&\leq \tau^2\fint_{B_{h_0}(x_0,\tau)} \tilde {\mathcal{R}}\, d\mathrm{vol}_{h_0}\\
        &=\tau^2\fint_{B_{h_0}(x_0,\tau)} \left(\mathcal{R}+2 n(n-1)\right)\, d\mathrm{vol}_{h_0}\\
        &\leq C_n k(x_0,\tau)+ 2 n(n-1) \tau^2.
    \end{split}
\end{equation}
We now estimate $u$ at $x\in B_{h_0}(x_0,\frac14 r_0)$. 

Then Lemma~\ref{lma:improved-heat}, Stokes' Theorem and co-area formula imply that if $t\leq (20\a)^{-2}\wedge (10^3(n-1)^2 \a^2)^{-1} \mathrm{diam}(M,h_0)^2$, then
\begin{equation}
    \begin{split}
        u(x,t)&\leq \int_{B_{h_0}(x,r_0)} \frac{C_1}{\mathrm{Vol}_{h_0}\left(B_{h_0}(x,\sqrt{t})\right)}\exp\left(-\frac{d_{h_0}^2(x,y)}{C_1t}\right) \tilde {\mathcal{R}}(y,0)  \, d\mathrm{vol}_{h_0}(y)\\
        &\leq \frac{C_1}{4} \frac{\mathrm{Vol}_{h_0}\left(B_{h_0}(x,r_0)\right)}{\mathrm{Vol}_{h_0}\left(B_{h_0}(x,\sqrt{t})\right)}\exp\left(-\frac{r_0^2}{C_1t} \right)\cdot  \left(C_n k(x,r_0)+8 n(n-1) \right) \\
        &\quad + \frac{C_n}{t}\int_0^{r_0} \frac{ k(x,r)}{r}\frac{\mathrm{Vol}_{h_0}\left(B_{h_0}(x,r)\right)}{\mathrm{Vol}_{h_0}\left(B_{h_0}(x,\sqrt{t})\right)}\exp\left(-\frac{r^2}{C_1t}\right)   \, dr\\
        &\quad +\frac{8 n(n-1)}t \int^{r_0}_0 \frac{\mathrm{Vol}_{h_0}\left(B_{h_0}(x,r)\right)}{\mathrm{Vol}_{h_0}\left(B_{h_0}(x,\sqrt{t})\right)}\exp\left(-\frac{r^2}{C_1t}\right)   \, dr\\
        &=\mathbf{I}+\mathbf{II}+\mathbf{III}.
    \end{split}
\end{equation}

If $\sqrt{t}\leq r_0$, then volume comparison implies 
\begin{equation}
    \begin{split}
        \mathbf{I}&\leq \frac{C_4}{t^{n/2}}\exp\left(-\frac{1}{C_4t} \right)\cdot\left( C_n k(x,r_0)+8 n(n-1) \right)\\
        &\leq C_4 (C_n\e_0+1)
    \end{split}
\end{equation}
and
\begin{equation}
    \begin{split}
        \mathbf{III}&\leq \frac{C_5 }{t} \left( \int^{r_0}_{\sqrt{t}} \frac{r^n}{t^{n/2}}\exp\left(-\frac{r^2}{C_1t}\right)   \, dr + \int^{\sqrt{t}}_0\exp\left(-\frac{r^2}{C_1t}\right)  dr \right)\\
        &\leq  \frac{C_5}{t} \left( \sqrt{t}\int^{\infty}_{1} r^n\exp\left(-\frac{r^2}{C_1}\right)   \, dr + \sqrt{t}\int^{1}_0\exp\left(-\frac{r^2}{C_1}\right)  dr \right)\\
        &\leq C_5  t^{-1/2}.
    \end{split}
\end{equation}

It remains to estimate the second term $\mathbf{II}$. This is the leading term when $t\to 0$. Using volume comparison as above to see that if $\sqrt{t}\leq r_0$, then 
\begin{equation}
    \begin{split}
    \mathbf{II}&=\frac{C_n}{t} \left(\int_{\sqrt{t}}^{r_0}+\int_0^{\sqrt{t}}\right) \frac{ k(x,r)}{r}\frac{\mathrm{Vol}_{h_0}\left(B_{h_0}(x,r)\right)}{\mathrm{Vol}_{h_0}\left(B_{h_0}(x,\sqrt{t})\right)}\exp\left(-\frac{r^2}{C_1t}\right)   \, dr\\
    &\leq \frac{C_6}t \int^{r_0}_0 \frac{k(x,r)}{r} \exp\left(-\frac{r^2}{C_6t}\right)   \, dr\\
    &\leq  \frac{C_6}{t} f(x,r_0)\leq C_6\e_0 t^{-1}.    \end{split}
\end{equation}

We now require $\e_0$ to be sufficiently small so that $\mathbf{II}\leq \frac12 t^{-1}$ and thus $u(x,t)\leq  t^{-1}$ if $t\leq \min\{ t_1,c_n,r_0^2,(20\a)^{-2}\}$ for some small $c_n>0$. By \cite[Theorem 1.1]{LeeTam2022}, there exists $S_1(n)>0$ such that if $$t\leq \min\{S_1r_0^2, t_1,c_n,r_0^2,(20\a)^{-2},c_n\a^{-1} \mathrm{diam}(M,h_0)^2\},$$ then 
\begin{equation}
    \begin{split}
       e^{-2n(n-1)t}\tilde{\mathcal{R}}(x,t)&= \varphi (x,t)\leq u(x,t)+v \\
        &\leq t^{-1}+(\frac14 r_0)^{-2}
    \end{split}
\end{equation}
at $x=x_0$ and $t\in[0,\min\{S_1r_0^2, t_1,c_n,r_0^2,(20\a)^{-2},c_n\a^{-1} \mathrm{diam}(M,h_0)^2\}]$. 
Since $x_0$ is arbitrary, by the choice of $\a$ and  \eqref{PIC1-control}, we see that  $t_1\geq S_2(n)>0$ if $M$ is complete non-compact. And if $M$ is compact, then $t_1\geq S_2(n)\cdot \mathrm{diam}(M,h_0)^2$. This completes the proof.
\end{proof}

\section{Existence of Ricci flow: general case}\label{Sec:Shorttime-General}
In this section, we will mainly consider the non-compact case and will construct a short-time solution $g(t)$ to the Ricci flow using pseudo-locality. We remark that the initial metric can a-priori be very complicated at infinity. In particular, the curvature is not necessarily bounded uniformly so that Shi's construction \cite{Shi1989} does not apply directly. To overcome this, we use a trick of Topping \cite{Topping2010} to construct local solution.  

\begin{lma}\label{lma:Exhaustion}
Suppose $(M^n,g_0)$ is a complete non-compact Riemannian manifold and $x_0\in M$ such that $K(g_0)\geq -L$ on $M$ for some $L>1$. Then there exist $C_n>0$ and a smooth positive proper function $\rho$ on $M$ such that $|\rho(x)-d_{g_0}(x,x_0)|\leq 1$, $|\nabla \rho|^2\leq 2$ on $M$ and $\nabla^2\rho \leq C_n L g_0$ outside $B_{g_0}(x_0,\sqrt{L})$.
\end{lma}
\begin{proof}
This follows from the standard Hessian comparison and the approximation method of Greene-Wu \cite{GreeneWu1973,GreeneWu1974,GreeneWu1979}. The linear dependence on $L$ can be seen from  scaling argument.
\end{proof}

We will use the smoothed distance function $\rho$ from Lemma~\ref{lma:Exhaustion} to construct a sequence of complete bounded curvature manifolds to approximate the original manifold $(M,g_0)$ in suitable sense. To do this, let $\kappa\in (0,1)$, $f:[0,1)\to[0,\infty)$ be the function:
\be\label{e-exh-1}
 f(s)=\left\{
  \begin{array}{ll}
    0, & \hbox{$s\in[0,1-\kappa]$;} \\
    -\displaystyle{\log \lf[1-\lf(\frac{ s-1+\kappa}{\kappa}\ri)^2\ri]}, & \hbox{$s\in (1-\kappa,1)$.}
  \end{array}
\right.
\ee
Let   $\varphi\ge0$ be a smooth function on $\R$ such that $\varphi(s)=0$ if $s\le 1-\kappa+\kappa^2 $, $\varphi(s)=1$ for $s\ge 1-\kappa+2 \kappa^2 $
\be\label{e-exh-2}
 \varphi(s)=\left\{
  \begin{array}{ll}
    0, & \hbox{$s\in[0,1-\kappa+\kappa^2]$;} \\
    1, & \hbox{$s\in (1-\kappa+2\kappa^2,1)$.}
  \end{array}
\right.
\ee
such that $\displaystyle{\frac2{ \kappa^2}}\ge\varphi'\ge0$. The function
 $$\mathfrak{F}(s):=\int_0^s\varphi(\tau)f'(\tau)d\tau.$$
satisfies the important properties.
\begin{lma}[Lemma 4.1 in \cite{LeeTam2020}] \label{l-exhaustion-1}Suppose   $0<\kappa<\frac18$. Then the function $\mathfrak{F}\ge0$ defined above is smooth and satisfies the following:
\begin{enumerate}
  \item [(i)] $\mathfrak{F}(s)=0$ for $0\le s\le 1-\kappa+\kappa^2$.
  \item [(ii)] $\mathfrak{F}'\ge0$ and for any $k\ge 1$, $\exp( -k\mathfrak{F})\mathfrak{F}^{(k)}$ is uniformly  bounded.
  \item [(iii)]  For any $ 1-2\kappa <s<1$, there is $\tau>0$ with $0<s -\tau<s +\tau<1$ such that
 \bee
 1\le \exp(\mathfrak{F}(s+\tau)-\mathfrak{F}(s-\tau))\le (1+c_2\kappa);\ \ \tau\exp(\mathfrak{F}(s-\tau))\ge c_3\kappa^2
 \eee
  for some absolute constants  $c_2>0, c_3>0$.
\end{enumerate}
\end{lma}

We now prove the main Theorem, Theorem~\ref{Thm:RF-existence}.
\begin{proof}[Proof of Theorem~\ref{Thm:RF-existence}]
We only consider the case $n\geq 4$ since the case of $n=3$ can be proved using similar but simpler argument. We will also focus on the non-compact case since the compact case follows directly from Theorem~\ref{thm:pseudo-locality}.

By scaling, we will assume $r=1$. To construct the approximation, take $\rho$ to be the smooth function obtained from Lemma~\ref{lma:Exhaustion}. For any $R>0$ sufficiently large, we let $U_R$ be the component of $\{x\in M: \rho(x)<R\}$ which contains a fixed point $x_0\in M$. In this way, $U_R$ will exhaust $M$ as $R\to +\infty$. Without loss of generality, we might assume $U_R$ to have smooth boundary. On each $U_R$, define 
$$F_R(x)=\mathfrak{F}\left(\frac{\rho(x)}{R}\right),\;\; g_{R,0}=e^{2F_R}g_0.$$

By \cite[Lemma 4.3]{LeeTam2020} (see also \cite{Hochard}), $g_{R,0}$ is a complete metric on $U_R$ with bounded curvature (in fact, bounded geometry of infinity order). We fix a small $\kappa$ in the construction of $\mathfrak{F}$. We claim that for all sufficiently large $R$, each $g_{R,0}$ satisfies the assumptions in Theorem~\ref{thm:pseudo-locality} regardless of how large $L$ is. We will omit the $R$ on $F_R$ for notational convenience.

\begin{claim}\label{claim:PIC1-preserved}
There is $R_0>0$ such that for all $R>R_0$, 
$$\mathrm{Rm}(g_{R,0})+g_{R,0}\owedge g_{R,0}\in \mathrm{C}_{\mathrm{PIC1}}.$$
\end{claim}
\begin{proof}[Proof of claim]
For notational convenience, We use $\tilde g,\tilde R$ to denote $g_{R,0}$ and the curvature tensor of $g_{R,0}$.  Recall  the curvature under conformal change:
    $$\widetilde {\mathrm{Rm}}=e^{2F}\mathrm{Rm}-e^{2F}g\owedge \left(\nabla^2F- dF\otimes dF+\frac12 |\nabla F|^2g \right).$$

Since $\mathrm{C}_{\mathrm{PIC1}}$ is convex, it suffices to estimate the lower bound  of each terms with respect to $\mathrm{C}_{\mathrm{PIC1}}$. By assumption,  $\mathrm{Rm}+r^{-2} g\owedge g/2\in \mathrm{C}_{\mathrm{PIC1}}$ so that 
\begin{equation}
e^{2F}\mathrm{Rm}+e^{-2F}r^{-2} \tilde g\owedge \tilde g/2\in \mathrm{C}_{\mathrm{PIC1}}.
\end{equation}
Let $\{\tilde e_i\}_{i=1}^4$ and $\{e_i=e^F\tilde e_i\}_{i=1}^4$ be an orthonormal frame w.r.t $\tilde g$ and $g$ respectively. Then for any $\lambda\in [0,1]$, 
\begin{eqnarray*}
& &\widetilde {\mathrm{Rm}}_{\t 1\t 3\t 3 \t 1}+(\tilde g\owedge \tilde g)_{\t 1\t 3\t 3 \t 1}+\lambda^2\widetilde {\mathrm{Rm}}_{\t 1\t 4\t 4 \t 1}+\lambda^2(\tilde g\owedge \tilde g)_{\t 1\t 4\t 4 \t 1}\\
& &+ \widetilde {\mathrm{Rm}}_{\t 2\t 3\t 3 \t 2}+(\tilde g\owedge \tilde g)_{\t 2\t 3\t 3 \t 2}+\lambda^2\widetilde {\mathrm{Rm}}_{\t 2\t 4\t 4 \t 2}+\lambda^2(\tilde g\owedge \tilde g)_{\t 2\t 4\t 4 \t 2}\\
& & +2\lambda \widetilde {\mathrm{Rm}}_{\t 1\t 2\t 3 \t 4}+2\lambda(\tilde g\owedge \tilde g)_{\t 1\t 2\t 3 \t 4}\\
&=&\widetilde {\mathrm{Rm}}_{\t 1\t 3\t 3 \t 1}+\lambda^2\widetilde {\mathrm{Rm}}_{\t 1\t 4\t 4 \t 1}+\widetilde {\mathrm{Rm}}_{\t 2\t 3\t 3 \t 2}+\lambda^2\widetilde {\mathrm{Rm}}_{\t 2\t 4\t 4 \t 2}+2\lambda \widetilde {\mathrm{Rm}}_{\t 1\t 2\t 3 \t 4}\\
& &+4(1+\lambda^2)\\
&=&e^{-2F}\left(\mathrm{Rm}_{1331}+\lambda^2\mathrm{Rm}_{1441}+\mathrm{Rm}_{2332}+\lambda^2\mathrm{Rm}_{2442}+2\lambda\mathrm{Rm}_{1234}\right)\\
& &+4(1+\lambda^2)+E_0\\
&\ge& (4-2e^{-2F})(1+\lambda^2)+E_0\\
&\ge& 2(1+\lambda^2)+E_0,
\end{eqnarray*}
where the error term $E_0$ is given by
\begin{eqnarray*}
&&-e^{-2F}\left((1+\lambda^2)(F_{11}+F_{22}-F_1F_1-F_2F_2)+2F_{33}-2 F_3F_3+2\lambda^2F_{44}-2\lambda^2F_4F_4\right)\\
&&-2e^{-2F}|\nabla F|^2(1+\lambda^2).
\end{eqnarray*}
We shall use Lemma \ref{l-exhaustion-1} to estimate the derivatives of $F$. By Lemma \ref{l-exhaustion-1} (ii)
\[
e^{-2F}|\nabla F|^2=e^{-2F}\left(\mathfrak{F}'\right)^2\frac{|\nabla \rho|^2}{R^2}\le C\frac{|\nabla \rho|^2}{R^2}\le \frac{2C}{R^2}.
\]
For large $R\gg 1$, $\mathfrak{F}'\frac{\rho_{ii}}{R}\leq \mathfrak{F}'\frac{C_nL}{R}$, hence by Lemmas \ref{lma:Exhaustion} and \ref{l-exhaustion-1} (ii), for $i=1,\dots, 4$,
\[
e^{-2F}F_{ii}=e^{-2F}\left(\mathfrak{F}'\frac{\rho_{ii}}{R}+\mathfrak{F}''\frac{\rho_i\rho_i}{R^2}\right)\le \frac{C_nL}{R}+\frac{C_n}{R^2},
\]
here we are not using Einstein summation convention, repeated indices are not summed.
From this we see that 
\[
E_0\ge -\frac{C_nL(1+\lambda^2)}{R}
\]
and for all sufficiently large $R$
\begin{eqnarray*}
& &\widetilde {\mathrm{Rm}}_{\t 1\t 3\t 3 \t 1}+(\tilde g\owedge \tilde g)_{\t 1\t 3\t 3 \t 1}+\lambda^2\widetilde {\mathrm{Rm}}_{\t 1\t 4\t 4 \t 1}+\lambda^2(\tilde g\owedge \tilde g)_{\t 1\t 4\t 4 \t 1}\\
& &+ \widetilde {\mathrm{Rm}}_{\t 2\t 3\t 3 \t 2}+(\tilde g\owedge \tilde g)_{\t 2\t 3\t 3 \t 2}+\lambda^2\widetilde {\mathrm{Rm}}_{\t 2\t 4\t 4 \t 2}+\lambda^2(\tilde g\owedge \tilde g)_{\t 2\t 4\t 4 \t 2}\\
& & +2\lambda \widetilde {\mathrm{Rm}}_{\t 1\t 2\t 3 \t 4}+\lambda(\tilde g\owedge \tilde g)_{\t 1\t 2\t 3 \t 4}\\
&\ge&(2-C_nLR^{-1})(1+\lambda^2)>0.
\end{eqnarray*}
This completes the proof of the claim.
\end{proof}

\begin{claim}\label{claim:Morrey-preserved}
There exists $C_n,R_0>0$ such that if $R>R_0$,  then for all $(x,r)\in U_R\times [0,1)$, 
\begin{equation}
f_{g_{R,0}}(x,r)=\displaystyle \int^r_0 s^{-1} k_{g_{R,0}}(x,s) \,ds \leq C_n \e_0 +C_nr.
\end{equation}
\end{claim}
\begin{proof}[Proof of Claim]
Let $x\in U_R$ and $r<1$. 
If $x$ is such that $\rho(x)<(1-2\kappa)R$, then for all $z\in B_{g_0}(x,1)$, $\rho(z)\leq (1-\kappa+\kappa^2)R$ provided that $R$ is sufficiently large. This particularly implies $g_{R,0}=g_0$ on $B_{g_0}(x,1)$ so that the conclusion holds trivially thanks to the assumptions.

If $(1-2\kappa)R\leq \rho(x)<R$,  then (iii) in Lemma~\ref{l-exhaustion-1} implies that as long as $R$ is sufficiently large,  we have $B_{g_{R,0}}(x,1)\subset  \{z\in U_R:  s-\tau<R^{-1}\rho(z)<s+\tau\}$ where $s=R^{-1}\rho(x)$ and $\tau=\tau(s)$ is a constant depending only on $s$. In particular, we have 
\begin{equation}\label{metric-equ}
e^{2\mathfrak{F}(s-\tau)}g_0\leq g_{R,0}\leq e^{2\mathfrak{F}(s+\tau)}g_0
\end{equation}
on $B_{g_{R,0}}(x,1)$. By the conformal change formula for scalar curvature, for all $r\in (0,1]$,
\begin{equation}
\begin{split}
&\quad r^2\fint_{B_{g_{R,0}}(x,r)}\mathcal{R}_{g_{R,0}} \, d\mathrm{vol}_{g_{R,0}}\\
&=r^2\fint_{B_{g_{R,0}}(x,r)} e^{-2F}\mathcal{R}_{g_{0}} \, d\mathrm{vol}_{g_{R,0}}-\frac{2(n-1)}{R}r^2\fint_{B_{g_{R,0}}(x,r)} e^{-2F}\mathfrak{F}'\Delta\rho \, d\mathrm{vol}_{g_{R,0}}\\
&\quad +r^2\fint_{B_{g_{R,0}}(x,r)} e^{-2F}\left[-\frac{4(n-1)}{n-2}\left(\frac{(n-2)^2}{4R^2}|\mathfrak{F}'|^2|\nabla\rho|^2+\frac{n-2}{2R^2}\mathfrak{F}''|\nabla\rho|^2\right)\right] \, d\mathrm{vol}_{g_{R,0}}\\
&=\mathbf{I}+\mathbf{II}+\mathbf{III}.
\end{split}
\end{equation}

By Lemma~\ref{lma:Exhaustion} and Lemma~\ref{l-exhaustion-1}, for all $r\in (0,1]$,
\begin{equation}
\begin{split}
\mathbf{III}&\leq \frac{C_n}{R^2}\cdot r^2\fint_{B_{g_{R,0}}(x,r)}\, d\mathrm{vol}_{g_{R,0}}\leq  \frac{C_nr^2}{R^2}=o(1)r^2.
\end{split}
\end{equation}

For $\mathbf{II}$, we control it using Stokes' Theorem, Lemma~\ref{l-exhaustion-1} and Lemma~\ref{lma:Exhaustion}:
\begin{equation}
\begin{split}
\mathbf{II}&=-\frac{2(n-1)}{R }\frac{r^2}{\mathrm{Vol}_{g_{R,0}}\left(B_{g_{R,0}}(x,r)\right)}\int_{B_{g_{R,0}}(x,r)} e^{(n-2)F}\mathfrak{F}'\Delta \rho \, d\mathrm{vol}_{g_{0}}\\
&=\frac{2(n-1)}{R }\frac{r^2}{\mathrm{Vol}_{g_{R,0}}\left(B_{g_{R,0}}(x,r)\right)}  \int_{B_{g_{R,0}}(x,r)} \nabla \rho\cdot  \nabla(e^{(n-2)F}\mathfrak{F}')\, d\mathrm{vol}_{g_{0}}\\
&\quad -\frac{2(n-1)}{R }\frac{r^2}{\mathrm{Vol}_{g_{R,0}}\left(B_{g_{R,0}}(x,r)\right)}\int_{\partial B_{g_{R,0}(x,r)}} e^{(n-2)F}\mathfrak{F}'  \cdot \nabla_\nu \rho \, dA_{g_0}\\
&\leq \frac{C_nr^2}{R^2}+\frac{C_n}{R }\frac{r^2}{\mathrm{Vol}_{g_{R,0}}\left(B_{g_{R,0}}(x,r)\right)}|\partial B_{g_{R,0}}(x,r)|_{g_{R,0}} \leq \frac{C_nr}{R^2}=o(1)r.
\end{split}
\end{equation}
Here we have used the volume comparison and Claim~\ref{claim:PIC1-preserved} on the last inequality. 

It remains to control $\mathbf{I}$.  Using \eqref{metric-equ}, Lemma~\ref{l-exhaustion-1} and volume comparison,  for $r\in (0,1]$,
\begin{equation}
\begin{split}
&\quad r^2\fint_{B_{g_{R,0}}(x,r)} e^{-2F} \mathcal{R}_{g_{0}} \, d\mathrm{vol}_{g_{R,0}}\\
&\leq \frac{C_nr^2}{\mathrm{Vol}_{g_{R,0}}\left(B_{g_{R,0}}(x,r)\right)} \int_{B_{g_{R,0}}(x,r)} e^{(n-2)F} |\Rm(g_0)| \, d\mathrm{vol}_{g_{0}}\\
&\leq \frac{C_n r^2 \cdot e^{(n-2)\mathfrak{F}(s+\tau)-n\mathfrak{F}(s-\tau)}}{\mathrm{Vol}_{g_{0}}\left(B_{g_{0}}(x,e^{-\mathfrak{F}(s+\tau)}r)\right)}  \int_{B_{g_{0}}(x,e^{-\mathfrak{F}(s-\tau)}r)} |\Rm(g_0)| \, d\mathrm{vol}_{g_{0}}\\
&\leq C_ne^{(n-2)(\mathfrak{F}(s+\tau)-\mathfrak{F}(s-\tau))} \frac{\mathrm{Vol}_{g_{0}}\left(B_{g_{0}}(x,e^{-\mathfrak{F}(s-\tau)}r)\right)} {\mathrm{Vol}_{g_{0}}\left(B_{g_{0}}(x,e^{-\mathfrak{F}(s+\tau)}r)\right)} \cdot  k(x,e^{-\mathfrak{F}(s-\tau)}r)\\
&\leq C_n k(x,e^{-\mathfrak{F}(s-\tau)}r).
\end{split}
\end{equation}

Therefore for each $(x,r)\in U_R\times [0,1)$,
\begin{equation}
   \begin{split}
        \int^r_0 w^{-1} k(x,e^{-\mathfrak{F}(s-\tau)}w) \,d w &= \int^{re^{-\mathfrak{F}(s-\tau)}}_0 w^{-1} k(x,w) \,d w\\
        &= f(re^{-\mathfrak{F}(s-\tau)})\leq  f(1)\leq  \e_0.
   \end{split}
\end{equation}

Hence for all $(x,r)\in U_R\times [0,1)$ and $R$ sufficiently large, 
\begin{equation}
 \displaystyle \int^r_0  s\left(\fint_{B_{g_{R,0}}(x,s)}\mathcal{R}_{g_{R,0}}  \, d\mathrm{vol}_{g_{R,0}}\right) dr\leq C_n\e_0 +o(1)r^2.
\end{equation}

The claim now follows from Claim~\ref{claim:PIC1-preserved} and the fact that $|\mathrm{R}|\leq \b_n |\mathrm{scal}(\mathrm{R})|$ for all $\mathrm{R}\in \mathrm{C}_{\mathrm{PIC1}}$.
\end{proof}

By Claim~\ref{claim:Morrey-preserved}, if we require $\e_0$ to be small enough, we see that for all $R\to+\infty$, $ \e_0^{-2}g_{R,0}$ satisfies the assumptions in Theorem~\ref{thm:pseudo-locality}. By re-scaling back, $U_R$ admits a short-time solution $g_R(t),t\in [0, \e_0^2 S_n]$ to the Ricci flow with $g_R(0)=g_{R,0}$ and 
\begin{enumerate}
    \item[(a)] $|\Rm_{g_R}(x,t)|\leq \a_n t^{-1}$;
    \item[(b)] $\mathrm{Rm}(g_R(t))+ \frac12 L_n\e_0^{-2}\, g_R(t)\owedge g_R(t)
    \in \mathrm{C}_{\mathrm{PIC1}}$.
\end{enumerate}
on $U_R\times (0,\e_0^2 S_n]$.  Since $g_{R,0}=g_0$ on any compact subset $\Omega\Subset M$ as $R\to +\infty$. By \cite[Corollary 3.2]{Chen2009} (see also \cite{Simon2008}) and the modified Shi's higher order estimates \cite[Theorem 14.16]{ChowBookII}, we infer that for any $m\in \mathbb{N}$ and $\Omega\Subset M$, we can find $C(n,m,\Omega,g_0)>0$ so that for all $R\to +\infty$,
\begin{align}
\sup_{\Omega\times [0,\hat \e_0 S_n]}|\nabla^m \mathrm{Rm}(g_R(t))|\leq C(n,m,\Omega,g_0).
\end{align}
By working on coordinate charts and  Ascoli-Arzel\`a theorem, we may pass to a subsequence to obtain a smooth solution $g(t)=\lim_{R\rightarrow +\infty}g_R(t)$ of the Ricci flow on $M\times [0,\e_0^2 S_n ]$ with $g(0)=g_0$ so that the estimates in conclusion holds. Moreover, it is a complete solution by \cite[Corollary 3.3]{SimonTopping2016}. This completes the proof by relabelling the constants.
\end{proof}

\section{application to Gap Theorem}
In this section, we will consider complete non-compact manifolds with non-negative complex sectional curvature and small integral average quadratic curvature decay.  Before we prove the main Theorem, we first prove Theorem~\ref{thm:gap-PIC1}, i.e. the gap Theorem under  Euclidean volume condition.  The relaxation of curvature condition is motivated by the recent work of Lott \cite{Lott2019}.

\begin{proof}[Proof of Theorem~\ref{thm:gap-PIC1}]
 By \cite[Theorem 1.1 \& Corollary 4.1]{HeLee2021} (see also \cite{Lai2019,SimonTopping2017}), $M$ is diffeomorphic to $\mathbb{R}^n$ and admits a long-time solution $g(t)$ to the Ricci flow on $M\times [0,+\infty)$ with $\mathrm{Rm}(g(t))\in \mathrm{C}_{\mathrm{PIC1}}$ and $|\Rm(g(t))|\leq \a t^{-1}$ for some $\a>0$.  It is well-known that the asymptotic volume ratio is preserved under Ricci flow with $\Ric(g(t))\geq 0$ and $|\Rm(g(t))|\leq \a t^{-1}$, for instances see the proof of \cite[Theorem 7]{Yokota2008}. If $n=3$, it also follows from \cite{LeeTam2022} that the non-negativity of sectional curvature is also preserved. Therefore, there exists $v_0>0$ such that for all $x\in M, r>0$ and $t>0$, we have
 \begin{equation}
 \label{AVR-all-t}    \mathrm{Vol}_{g(t)}\left( B_{g(t)}(x,r)\right)\geq v_0 r^n.
 \end{equation}
 
By Lemma~\ref{lma:R-improve}, $\rheat \mathcal{R}\leq \mathcal{R}^2$ and hence the maximum principle implies that for all $(x,t)\in M\times (0,+\infty)$,
\begin{equation}
    \mathcal{R}(x,t)\leq \int_M G(x,t;y,0) \mathcal{R}(y,0)\, d\mathrm{vol}_{g_0}(y)
\end{equation}
where $G(x,t;y,s)$ denotes the heat kernel on $M$. The maximum principle can be justified by applying it on $[s,t]$ for $s>0$ and followed by passing $s\to 0$ or applying the localized maximum principle \cite[Theorem 1.1]{LeeTam2022} as in the proof of Theorem~\ref{thm:pseudo-locality} and followed by exhaustion argument.  
By \eqref{AVR-all-t} and Lemma~\ref{lma:improved-heat}, $G(x,t;y,s)$ satisfies the Gaussian  estimate (see also 
\cite[Proposition 3.1]{BamlerCabezasWilking2019}):
\begin{equation}
    G(x,t;y,0)\leq \frac{C(n,v_0,\a)}{t^{n/2}} \exp\left(-\frac{d_{g_0}(x,y)^2}{C(n,v_0,\a)t} \right)
\end{equation}
for all $x,y\in M$ and $t>0$. Fix $x_0\in M$ and $t$ sufficiently large. In what follows, we will use $C_i$ to denote constants depending only on $n,v_0,\a$.

Then, argue as in the proof of Theorem~\ref{thm:pseudo-locality}. Stokes' Theorem and co-area formula imply
\begin{equation}
    \begin{split}
    \mathcal{R}(x_0,t)&\leq \int_M G(x_0,t;y,0)\, \mathcal{R}(y,0)\,d\mathrm{vol}_{g_0}(y)\\
    &\leq \int^\infty_0  \frac{ r^n}{t^{n/2}}\cdot \exp\left(-\frac{r^2}{C_1t} \right)\frac{C_1}{rt} k(x_0,r) dr.
    \end{split}
\end{equation}

By assumption, for any $\delta>0$, there exists $r_0>0$ such that for all $r>r_0$, $k(x_0,r)<\delta$. We split the integral using $r_0$.
\begin{equation}
    \begin{split}
 &\quad \int^\infty_0  \frac{ r^n}{t^{n/2}}\cdot \exp\left(-\frac{r^2}{C_1t} \right)\frac{C_1}{rt} k(x_0,r)  dr\\
 &=\left(\int^\infty_{r_0}+ \int^{r_0}_0 \right) \frac{ r^n}{t^{n/2}}\cdot \exp\left(-\frac{r^2}{C_1t} \right)\frac{C_1}{rt} k(x_0,r)  dr\\
 &=\mathbf{I}+\mathbf{II}.
    \end{split}
\end{equation}

We let $\Lambda(g_0,r_0,\delta)>0$ be such that $\mathcal{R}_{g_0}\leq \Lambda$ on $B_{g_0}(x_0,r_0)$ and hence $k(x_0,r)\leq C_{r_0}r^2$ for all $r\leq r_0$.  Therefore, 
\begin{equation}
    \begin{split}
        \mathbf{II}&\leq C_1 \Lambda\int_0^{r_0}\frac{ r^n}{t^{n/2}} \cdot \exp\left(-\frac{r^2}{C_1t} \right)\frac{r}{t}  dr\\
        &=C_1\Lambda \int^{r_0t^{-1/2}}_0 r^{n+1} dr\\
        &\leq C'(n,r_0,v_0,g_0,\delta ) t^{-1-\frac{n}2}
    \end{split}
\end{equation}
while
\begin{equation}
    \begin{split}
        \mathbf{I}&\leq C_1\delta \int^\infty_{r_0} \frac{ r^n}{t^{n/2}}\cdot \exp\left(-\frac{r^2}{C_1t} \right)\frac{1}{rt} dr\\
        &\leq  t^{-1} C_1\delta \int^\infty_{0} r^{n-1}  \exp\left(-\frac{r^2}{C_1} \right) dr\\
        &\leq C_2t^{-1}\delta. 
    \end{split}
\end{equation}
By letting $t\to+\infty$ and followed by $\delta\to 0$ and the fact that $\mathrm{Rm}(g(t))\in \mathrm{C}_{\mathrm{PIC1}}$,  we conclude
\begin{equation}\label{cur-at-limit}
    \limsup_{t\to +\infty} \, t|\Rm(x_0,t)|=0
\end{equation}

Now consider the re-scaled Ricci flow $g_i(t),t\in [0,+\infty)$ where $g_i(t)=i^{-2}g(i^2 t)$. By Hamilton's compactness \cite{Hamilton1995}, $(M,g_i(t),x_0)$ sub-converges to $(M_\infty,g_\infty(t),x_\infty)$ for $t\in (0,+\infty)$ in the pointed $C^\infty$ Cheeger-Gromov sense as $i\to +\infty$. By the proof of \cite[Theorem 1.2]{SchulzeSimon2013} (which is based on Cheeger and Colding's volume continuity \cite{CheegerColding1997}), $g_\infty(t)$ satisfies $\mathrm{AVR}(g_\infty(t))=\mathrm{AVR}(g_0)$ and $\mathrm{Rm}(g_\infty(t))\in \mathrm{C}_{\mathrm{PIC1}}$ for all $t>0$. Moreover, \eqref{cur-at-limit} implies that $\mathcal{R}(g_\infty(x_\infty,t))=0$ for all $t>0$. The strong maximum principle implies $g_\infty(t)$ is Ricci-flat and hence flat globally on $M_\infty$. This forces $\mathrm{AVR}(g_\infty(t))=\mathrm{AVR}(g_0)=1$ and hence $(M,g_0)$ is flat Euclidean by the rigidity of volume comparison.
\end{proof}

Now we study the case with strong curvature condition but \textbf{without} volume growth assumption. We start with the long-time existence under a slightly weaker condition.
\begin{prop}\label{prop:LongTime}There exists $\e_0(n)>0$ such that the following holds: 
Let $(M^n,g_0)$ be a complete non-compact manifold such that $n\geq 3$, and 
\begin{enumerate}
\item[(i)] $\inf_M \mathrm{K}(g_0)>-\infty$; 
    \item[(ii)] $\mathrm{Rm}(g_0)\in \mathrm{C}_{\mathrm{PIC1}}$ if $n\geq 4$; 
    \item[(iii)] $\mathrm{K}(g_0)\geq 0$ if $n=3$.
\end{enumerate}
Suppose for all $x\in M$, 
       $$\int^{+\infty}_0 s\left(\fint_{B_{g_0}(x,s)}|\Rm(g_0)|\, d\mathrm{vol}_{g_0}\, \right)ds <\e_0.$$
    Then there exist $\a_n>0$ and a long-time solution $g(t)$ to the Ricci flow on $M\times [0,+\infty)$ with $g(0)=g_0$ such that $|\Rm(g(t))|\leq \a_n  t^{-1}$ and $\mathrm{Rm}(g(t))\in \mathrm{C}_{\mathrm{PIC1}}$ for all $t>0$. If in addition, $\mathrm{K}^\mathbb{C}(g_0)\geq 0$, then $\mathrm{K}^\mathbb{C}(g(t))\geq 0$ for all $t>0.$
\end{prop}
\begin{proof}
By Theorem~\ref{Thm:RF-existence}, for any $i>0$, there exists a solution $g_i(t)$ to the Ricci flow on $M\times [0,S_ni^2]$ with $g_{i}(0)=g_0$ and $|\Rm(g_i(t))|\leq \a_n t^{-1}$. Using the sub-sequential convergence argument in the proof of Theorem~\ref{Thm:RF-existence}, $g_i(t)$ sub-converges uniformly locally in $C^\infty_{loc}$ to $g(t)$ on $M\times [0,+\infty)$. This proves the existence part. The assertion of $\mathrm{Rm}(g(t))\in \mathrm{C}_{\mathrm{PIC1}}$ and $\mathrm{K}^\mathbb{C}(g(t))\geq 0$ follow from \cite[Theorem 3.1]{LeeTam2022}.
\end{proof}

\begin{rem}
    The lower bound of the sectional curvature is purely for technical reason. It is easy to see that it can be further relaxed to quadratic sectional lower bound. We expect that it can be removed in full generality. 
\end{rem}

We now finish the proof of Theorem~\ref{Thm:GapTheorem-higherD}.
\begin{proof}[Proof of Theorem~\ref{Thm:GapTheorem-higherD}]
By assumption and integrability, for any $\delta>0$, there exists $r_0>1$ such that for all $r>r_0$,
$$\int^\infty_{r} s^{-1}k(x_0,s)\,ds <\delta.$$
We also let $\Lambda(g_0,r_0,\delta)>0$ be large constant such that $\mathcal{R}\leq \Lambda$ on $B_{g_0}(x_0,r_0)$. We will also use $C_i$ to denote any dimensional constants.

We let $g(t)$ be the long-time solution obtained from Proposition~\ref{prop:LongTime}. Since $\mathrm{K}^\mathbb{C}(g(t))\geq 0$,  Lemma~\ref{lma:R-improve} implies $\rheat \mathcal{R}\leq \mathcal{R}^2$ and hence for all $t>0$,
\begin{equation}
\mathcal{R}(x_0,t)\leq \int_M G(x_0,t;y,0)  \mathcal{R}(y,0)\, d\mathrm{vol}_{g_0}(y)
\end{equation}
as in the proof of Theorem~\ref{thm:gap-PIC1}.  By using Lemma~\ref{lma:improved-heat}, Stokes' Theorem and co-area formula implies that for all $\sqrt{t}>r_0$,
\begin{equation}
\begin{split}
\mathcal{R}(x_0,t)&\leq \int_M  \frac{C_1}{\mathrm{Vol}_{g_0} \left( B_{g_0}(x_0,\sqrt{t})\right)}\exp\left(-\frac{d_{g_0}(x,y)^2}{C_1t} \right)\mathcal{R}(y,0)\, d\mathrm{vol}_{g_0}(y)\\
&=\frac{C_n}{t}\left(\int^\infty_{\sqrt{t}}+\int^{\sqrt{t}}_{r_0}+\int^{r_0}_0   \right)\frac{\mathrm{Vol}_{g_0} \left( B_{g_0}(x_0,r)\right)}{\mathrm{Vol}_{g_0} \left( B_{g_0}(x_0,\sqrt{t})\right)} \exp\left(-\frac{r^2}{C_1t} \right)\cdot \frac{k(x_0,r)}{r}\,dr\\
&= \mathbf{I}+\mathbf{II}+\mathbf{III}.
\end{split}
\end{equation}

When $r>\sqrt{t}$, we can appeal to volume comparison to deduce that 
\begin{equation}
\begin{split}
\mathbf{I}&\leq \frac{C_2}{t} \int^\infty_{\sqrt{t}} \left(\frac{r}{\sqrt{t}} \right)^n \exp\left(-\frac{r^2}{C_1t} \right)\cdot \frac{k(x_0,r)}{r}\,dr\leq \frac{C_3}{t} \int^\infty_{\sqrt{t}} \frac{k(x_0,r)}{r}\,dr.
\end{split}
\end{equation}

When $r\in [r_0,\sqrt{t}]$, we argue similarly:
\begin{equation}
\begin{split}
\mathbf{II}&\leq \frac{C_n }{t} \int^{\sqrt{t}}_{r_0}\exp\left(-\frac{r^2}{C_1t} \right) \frac{k(x_0,r)}{r} \,dr \leq \frac{C_4}{t} \int^{\sqrt{t}}_{r_0}\frac{k(x_0,r)}{r}\,dr
\end{split}
\end{equation}
so that $\mathbf{I}+\mathbf{II}\leq C_5 \delta t^{-1}$.

When $r\leq r_0$, we use the estimate $k(x_0,r)\leq \Lambda r^2$ and Yau's linear volume growth to deduce that   
\begin{equation}
\begin{split}
\mathbf{III}&\leq \frac{C_6\Lambda }{t} \int^{r_0}_0 \frac{\mathrm{Vol}_{g_0} \left( B_{g_0}(x_0,r)\right)}{\mathrm{Vol}_{g_0} \left( B_{g_0}(x_0,\sqrt{t})\right)} \exp\left(-\frac{r^2}{C_1t} \right)r \,dr\\
&\leq \frac{C_7 r_0^{n+2} \Lambda t^{-1-\frac12}  }{{\mathrm{Vol}_{g_0} \left( B_{g_0}(x_0,1)\right)} }  .
\end{split}
\end{equation}

Therefore,  for all $t\to +\infty$,
\begin{equation}
 t\cdot \mathcal{R}(x_0,t)\leq  C_5 \delta  + \frac{C_7r_0^{n+2} \Lambda t^{-\frac12}  }{{\mathrm{Vol}_{g_0} \left( B_{g_0}(x_0,1)\right)} }.
\end{equation}

By Brendle's Harnack inequality \cite{Brendle2009} and letting $\delta\to 0$, we deduce that for all $t>0$, $\mathcal{R}(x_0,t)=0$.  Since $\mathrm{K}^\mathbb{C}(g(t))\geq 0$, strong maximum principle implies that $g(t)$ is flat for all $t>0$ and hence $g_0$ is flat. This completes the proof.
\end{proof}

\section{Gap Theorem in dimension three}

The three dimensional case with only $\Ric\geq 0$ is more challenging due to the failure of Lemma~\ref{lma:R-improve}. In the presence of point-wise curvature decay, we will combine it with the methods in \K geometry to prove an optimal gap result.

\begin{proof}[Proof of Theorem~\ref{thm:3D-gap}]

By the work of Liu \cite{Liu2013}, either $M$ is diffeomorphic to $\mathbb{R}^3$ or the universal cover of $(M^3,g_0)$ splits isometrically as $\Sigma^2\times \mathbb{R}$ such that $\Sigma$ has non-negative sectional curvature. 
In case $M$ is not topologically Euclidean. 
Since $g_0$ has quadratic curvature decay at infinity, the universal cover cannot split isometrically as $\mathbb{S}^2\times \mathbb{R}$. 
{
We may argue as follows: Let $\e>0$ such that $K_0:=\{\mathcal{R}_M\ge \e\}$ is a nonempty compact set of $M$ and $\pi:\mathbb{S}^2\times \mathbb{R}\to M$ be a Riemannian covering map. Suppose $(a,z_0)\in \mathbb{S}^2\times \mathbb{R}$ and $\pi(a,z_0)\in K_0$. Then for any $w\in  \mathbb{R}, b\in \mathbb{S}^2$, $\pi(a,w)\in K_0$ and 
\begin{eqnarray*}
d_M(\pi(b,w), \pi(a,z_0))&\leq& \text{diam}_M(K_0)+d_M(\pi(b,w), \pi(a,w))\\
&\leq& \text{diam}_M(K_0)+d_{\mathbb{S}^2}(a,b)\\
&\leq& \text{diam}_M(K_0)+\text{diam}_{\mathbb{S}^2}(\mathbb{S}^2).
\end{eqnarray*}
This implies that $M$ has finite diameter. Contradiction.}
Therefore, $\Sigma^2$ must be topologically $\mathbb{R}^2$. We now use an argument originated from \K geometry. By \cite[Theorem 1.1, Theorem 1.2 and Proposition 1.1]{NiShiTam},  there exists $u\in C^\infty_{loc}(M)$ such that $\Delta_M u=\mathcal{R}_M\geq 0$ on $M$ with $u=o(\log r)$ as $r\to +\infty$. By \cite[Theorem 1.3]{NiShiTam}, $|\nabla u|$ is uniformly bounded on $M$. We lift $u$ to $\tilde M$ so that we might assume $\Delta_{\tilde M}u=\mathcal{R}_{\tilde M}$ on $\tilde M$ and $u$ is still sub-logarithm growth on $\tilde M$ since the covering map is distance non-increasing. 
For $(z,y)\in \Sigma\times \mathbb{R}=\tilde M$, since $\tilde M$ splits isometrically, $\partial_y u$ is bounded harmonic function on $\tilde M$. Hence the Cheng-Yau's gradient estimate implies $\partial_y\partial_y u=0$ on $\tilde M$. In particular, $\Delta_\Sigma u(z,y_0)=\mathcal{R}_\Sigma(z)$ for any fixed $y_0\in \mathbb{R}$. 

Since $\Sigma^2$ is two-dimensional with non-negative curvature, the Laplacian comparison implies that for a fixed $z_0\in \Sigma$, $r_\Sigma(\cdot)=d_{\Sigma}(\cdot,z_0)$ satisfies $\Delta_{\Sigma} \log r_\Sigma\leq 0$ in the sense of barrier for $r_\Sigma>0$. By applying the maximum principle to $u(z,y_0)-\e \log r_\Sigma $ for $\e\to 0^+$, we conclude that $u(z,y_0)$ attains its maximum at interior of $\Sigma$ and hence it must be constant by the strong maximum principle. Therefore, $\mathcal{R}_\Sigma=\mathcal{R}_M=0$. It follows that $M$ is flat in this case since its sectional curvature is non-negative. 

It remains to consider the case of $M^3=\mathbb{R}^3$ topologically. We first appeal to the work of Reiris \cite[Theorem 1.1]{Reiris2015} to see that $g_0$ must be of Euclidean volume growth. In particular,  as in the proof of Theorem~\ref{thm:gap-PIC1},  there exist $v,\a>0$ and a long-time solution $g(t)$ on $M\times[0,+\infty)$ such that for all $(x,t)\in M\times (0,+\infty)$,
\begin{enumerate}
\item[(i)] $|\Rm(g(x,t))|\leq \a t^{-1}$;
\item[(ii)] $\mathrm{Vol}_{g(t)}\left( B_{g(t)}(x,r)\right)\geq v_0 r^3$ for all $r>0$;
\item[(iii)] $\Ric(g(x,t))\geq 0$.
\end{enumerate}
The properties of $g(t)$ are scaling invariant.

We consider the blow-down Ricci flow as usual. Let $R_i\to +\infty$ and $g_i(t)=R_i^{-2}g(R_i^2 t)$ so that  $(M_i^3,g_i(t),x_i)=(M,g_i(t),x_0)$ sub-converges to $(M_\infty,g_\infty(t),x_\infty)$ in the smooth pointed Cheeger-Gromov sense for $t\in (0,+\infty)$ as $i\to +\infty$.  Moreover,  the limiting Ricci flow $g_\infty(t)$ satisfies $\mathrm{AVR}(g_\infty(t))=\mathrm{AVR}(g_0)>0$ for all $t>0$.  Clearly, $M_\infty$ is still topologically $\mathbb{R}^3$. Our goal is to show that the asymptotic volume ratio satisfies $\mathrm{AVR}(g_\infty(t))=1$.  In what follow, we will use $C_i$ to denote constant depending only on $\a,v_0$.

For $r>1$ sufficiently large.  Let $z_\infty \in M_\infty$ be such that $d_{g_\infty(1)}(x_\infty, z_\infty)=r$. By the smooth convergence and distance distortion \cite[Corollary 3.3]{SimonTopping2016},  there exists $z_i\in M$ with $d_{g_0}(x_0,z_i)\in (\frac12 R_i r, 2 R_i r)$ such that $\mathcal{R}(g_\infty(z_\infty,1))$ is the limit of $\mathcal{R}(g_i(z_i,1))$.  Denote $g_{i,0}=R_i^{-2}g_0$ for notation convenience. On $B_{g_{i,0}}(z_i, \frac12 r)$,  assumption implies $|\Rm(g_{i,0})|\leq 4C_0r^{-2}$.  Moreover,  by volume comparison for $z\in B_{g_{i,0}}(z_i,\frac12 r)$ and $0<s\leq \frac12 r$,
\begin{equation}
\begin{split}
k_{g_{i,0}}(z,s)
&=(R_is)^2\fint_{B_{g_0}(z,R_is)} |\Rm(g_0)|\, d\mathrm{vol}_{g_0}\\
&\leq (R_is)^2\cdot \left(\frac{s+4r}{s} \right)^n \fint_{B_{g_0}(x_0,R_i(s+2r))} |\Rm(g_0)|\, d\mathrm{vol}_{g_0}\\
&\leq  \left(\frac{s+4r}{s} \right)^n \cdot  k(x_0,R_i(s+2r)).
\end{split}
\end{equation}

Using the assumption on asymptotic of $k(x_0,\cdot)$, for any $\sigma>0$, there exists $N\in \mathbb{N}$ such that for all $i>N$, we have 
\begin{equation}\label{2bdd of k}
\begin{split}
k_{g_{i,0}}(z,s) &\leq  \min\left\{ \left(\frac{4r+s}{s} \right)^n \sigma, 4C_0 s^2r^{-2} \right\}
\end{split}
\end{equation}
for all $z\in B_{g_{i,0}}(z_i,\frac12 r)$ and $0<s\leq \frac12 r$.

By \cite[Corollary 4.1]{LeeTam2022} and scaling argument, there exist 
$S_1,C_1>0$ such that 
$$\mathrm{K}(g_i(t))\geq -C_1r^{-2}$$ on $B_{g_{i,0}}(z_i,\frac14r)\times [0,S_1r^2]$.  We now intend to improve the lower bound of the sectional curvature.  Let $\phi(x,t)$ be the negative part of the lowest eigenvalue of $\mathrm{Rm}(g_i(x,t))$ with respect to the cone of non-negative sectional curvature. It is known that $\rheat \phi \leq \mathcal{R}\phi +c_0 \phi^2$ for some dimensional constant $c_0>0$ by \cite{BamlerCabezasWilking2019} and hence 
$$\rheat  (e^{-C_1c_0r^{-2}t}\phi) \leq \mathcal{R} (e^{-C_1c_0r^{-2}t}\phi) $$
on $B_{g_{i,0}}(z_i,\frac14r)\times [0,S_1r^2]$. We now apply \cite[Corollary 3.1]{LeeTam2022} to conclude that,  for any $\ell>0$ there exists $S_\ell>0$ such that 
\begin{equation}
e^{-c_0C_1r^{-2}t}\phi(x,t)-\int_{B_{g_{i,0}}(z_i,\frac14r)} G_i(x,t;y,0) \phi(y,0) \,d\mathrm{vol}_{g_{i,0}}(y)\leq  16^{\ell+1}t^\ell r^{-2(\ell+1)}
\end{equation}
on $B_{g_{i,0}}(z_i,\frac18 r)\times [0,S_\ell r^2]$, 
where $G_i$ denotes the dirichlet heat kernel of $g_i(t)$ on $B_{g_{i,0}}(z_i,\frac14 r)$.  
By the properties of $g_i(t)$ and Lemma~\ref{lma:improved-heat}, $G_i(x,t;y,s)$ satisfies a 
uniform Gaussian estimate and hence for $x\in B_{g_{i,0}}(z_i,\frac18r)$,
\begin{equation}\label{phi est}
\begin{split}
&\quad \int_{B_{g_{i,0}}(z_i,\frac14r)} G_i(x,t;y,0) \phi(y,0) \,d\mathrm{vol}_{g_{i,0}}(y)\\
& \leq  \sigma^\e \int^{\frac38 r}_0  \frac{C_2s^{n-1}}{t^{1+\frac{n}{2}}} \exp\left(-\frac{s^2}{C_2t} \right)  \left(\frac{4r+s}{s} \right)^{n\e} s^{2(1-\e)} r^{-2(1-\e)}ds\\
&\quad +\frac{C_2r^{n-2}}{t^{\frac{n}{2}}}\exp\left(-\frac{r^2}{C_2t} \right)k_{g_{i,0}}\left(z_i,\frac{3}{8}r\right)\\
&\leq C_3t^{-\e (1+\frac12 n)} r^{-2(1-\e)+n\e} \sigma^\e \int^{+\infty}_0  s^{n-1+2(1-\e)-n\e}\exp\left(-\frac{s^2}{C_2} \right) ds\\
&\quad +C_3r^{-2}\sigma
\end{split}
\end{equation}
where $\e$ is arbitrary in $(0,1)$ and $n=3$ now. Here we have used the bounds of $k_{g_{i,0}}(x,\cdot)$ from \eqref{2bdd of k}, interpolating between both of them. If we choose $\e$ sufficiently small so that $\e(1+\frac{n}{2} )<\frac12$ and then we see that the sectional curvature of $g_i(t)$ is bounded from below by 
$$\mathrm{K}(g_i(t))\geq -L(t)=-C_4\left(16^{\ell+1} t^{\ell} r^{-2(\ell+1)} + \sigma^\e t^{-\e( 1+\frac12 n)} r^{-2(1-\e)+n\e}+r^{-2}\sigma\right)$$
on $B_{g_{i,0}}(z_i,\frac18 r) \times [0,S_\ell r^2]$. 

Now we can proceed in a similar way as in the proof of Theorem~\ref{thm:pseudo-locality}.  Since $\mathrm{Rm}(g_i(t))+L\,g_i(t)\owedge g_i(t)/2$ is in the cone of non-negative sectional curvature,  the scalar curvature $\mathcal{R}$ satisfies
\begin{equation}
\begin{split}
\rheat \mathcal{R}&= 2|\widetilde{Ric}-2Lg|^2\\
&\leq  \mathcal{R}^2 +8 L \mathcal{R} +36L^2.
\end{split}
\end{equation}
In particular, $\varphi=\exp(-8\int^t_0 L(s)\,ds) \cdot \mathcal{R}$ satisfies
\begin{equation}
    \rheat \varphi \leq \mathcal{R} \varphi +36L^2
\end{equation}
on $B_{g_{i,0}}(z_i,\frac18 r) \times [0,S_\ell r^2]$. 

Define 
\begin{equation}
    \begin{split}
        u(x,t)&=\int_{B_{g_{i,0}}(z_i,\frac14r)}G_i(x,t;y,0)\mathcal{R}(y,0)\,d\mathrm{vol}_{g_{i,0}}(y)\\
        &\quad +36\int^t_0\int_{B_{g_{i,0}}(z_i,\frac14r)} G_i(x,t;y,s)L^2(s)\, d\mathrm{vol}_{g_{i}(s)}(y) ds
    \end{split}
\end{equation}
so that $\rheat u=\mathcal{R} u +36L^2$ and hence \cite[Corollary 3.1]{LeeTam2022} shows that by shrinking $S_\ell$ if necessary, for all $t\in[0,S_\ell r^2]$, 
\begin{equation}
    \varphi(z_i,t)=\exp\left(-8\int^t_0 L(s)\,ds\right) \cdot \mathcal{R}_{g_i(t)}(z_i)\leq u(z_i,t)+32^{\ell+1}t^{\ell}r^{-2(\ell+1)}
\end{equation}

We can use \cite[Proposition 4.1]{LeeTam2022}  
or Lemma~\ref{lma:improved-heat}, and argue in the same way as in \eqref{phi est} to see that
\begin{equation}
    \begin{split}
&\int_{B_{g_{i,0}}(z_i,\frac14r)}G_i(z_i,t;y,0)\mathcal{R}(y,0)\,d\mathrm{vol}_{g_{i,0}}(y)
\\
&\leq \int_0^{\frac{r}{4}} \frac{C}{t^{\frac{n}{2}}}\exp\left(-\frac{s^2}{C t} \right)\int_{\partial B_{g_{i,0}}(z_i,\frac14s)}\mathcal{R}\, ds\\
&=\frac{Cr^{n-2}}{t^{\frac{n}{2}}}\exp\left(-\frac{r^2}{Ct} \right)k_{g_{i,0}}\left(z_i,\frac{r}{4}\right)\\
&\quad+\int^{\frac14 r}_0  \frac{C_2s^{n-1}}{t^{1+\frac{n}{2}}} \exp\left(-\frac{s^2}{C_2t} \right) k_{g_{i,0}}\left(z_i,s\right)ds\\
&\leq C r^{-2}\sigma+\sigma^\e \int^{\frac14 r}_0  \frac{C_2s^{n-1}}{t^{1+\frac{n}{2}}} \exp\left(-\frac{s^2}{C_2t} \right)  \left(\frac{4r+s}{s} \right)^{n\e} s^{2(1-\e)} r^{-2(1-\e)}ds\\
&\leq C_5 \left(r^{-2}\sigma+\sigma^\e t^{-\e( 1+\frac12 n)} r^{-2(1-\e)+n\e}\right).
    \end{split}
\end{equation}

Similarly, by \cite[Proposition 4.1]{LeeTam2022} 
\begin{equation}
    \begin{split}
\int_{B_{g_{i,0}}(z_i,\frac14r)} G_i(z_i,t;y,s)\, d\mathrm{vol}_{g_{i}(s)}(y)&\leq \int_0^{\infty}\frac{C'w^{n-1}}{(t-s)^{\frac{n}{2}}}\exp\left(-\frac{w^2}{C(t-s)}\right)\,dw\\
&=C'\int_0^{\infty}\tau^{n-1}\exp\left(-\frac{\tau^2}{C}\right)\,d\tau
\leq C_6.
\end{split}
\end{equation}
Hence
\begin{equation}
    \begin{split}
    &\quad 36\int^t_0\int_{B_{g_{i,0}}(z_i,\frac14r)} G_i(x,t;y,s)L^2(s)\, d\mathrm{vol}_{g_{i}(s)}(y)\, ds\\
    &\leq C\int_0^{t}L^2(s)\,ds\\
    &\leq C_7\left(t^{2l+1}r^{-4(l+1)}+\sigma^{2\e}t^{1-2\e( 1+\frac12 n)} r^{-4(1-\e)+2n\e}+r^{-4}\sigma^2t\right).
    \end{split}
\end{equation}

We fix $t=1$ and $\ell=10$. We might assume $r$ to be large enough so that $1\in [0,S_\ell r^2]$ and hence the estimates applies at $t=1$. Also note that $\e$ is sufficiently small so that we have 
\[
\int_0^{1}L(s)\,ds\leq C\left(t^{l+1}r^{-2(l+1)}+\sigma^{\e}t^{1-\e( 1+\frac12 n)}r^{-2(1-\e)+n\e}+r^{-2}\sigma t\right)\Big|_{t=1}\leq C_8
\]
and 
\begin{equation}
    \begin{split}
    \varphi(z_i,1)\leq C_9r^{-2}\left(r^{-2\ell}+\sigma+\sigma^\e  r^{(n+2)\e}+r^{-2(2l+1)}+\sigma^{2\e} r^{-2(1-2\e)+2n\e}+r^{-2}\sigma^2\right).
    \end{split}
\end{equation}

Since $r$ is any arbitrarily large number and $\e$ is a small positive number, we conclude after passing $i\to +\infty$ and followed by $\sigma\to 0^+$
that $g_\infty(1)$ has $\Ric\geq 0$ and satisfies 
\begin{equation}
    |\mathrm{Rm}(g_\infty(z_\infty,1))| \leq C_{10} \left(d_{g_\infty(1)}(x_\infty,z_\infty) \right)^{-22}
\end{equation}
for $z_\infty \to +\infty$. Now a result of Greene-Petersen-Zhu  \cite[Theorem 2]{GreenePetersenZhu1994} can be applied to show that $g_\infty(1)$ is indeed flat since $M_\infty$ is topologically $\mathbb{R}^3$.  As $\mathrm{AVR}(g_\infty(1))>0$, $M_\infty$ is isometric to $\mathbb{R}^3$ and hence $\mathrm{AVR}(g_\infty(1))=\mathrm{AVR}(g_0)=1$. Now the flatness of $(M, g_0)$ follows from the rigidity case of the volume comparison theorem.  Alternatively, one can blow-down the manifold again together with the fast curvature decay to deduce flatness by pseudo-locality. This completes the proof. 
\end{proof}

\section{application to almost flatness}

In this  section, we discuss another application of the Ricci flow smoothing. We consider compact manifolds with almost vanishing curvature in the sense of integral average curvature. This is motivated by the recent work of Chen-Wei-Ye \cite{ChenWeiYe2022}.

\begin{proof}[Proof of Theorem~\ref{thm:almost-flat}]
By re-scaling, we assume $\mathrm{diam}(M,g_0)=1$. We will use the Ricci flow smoothing to find a better metric from $g_0$ using pseudo-locality Theorem. We only consider the case $n\geq 4$ while the case of $n=3$ can be done similarly as by strengthening the cone from $\mathrm{C}_{\mathrm{PIC1}}$ to cone of non-negative sectional curvature. 

We let $\e_1,\e_2>0$ be small dimensional constants to be determined. We will determine $\e_i$ so that metric $g_0$ satisfying 
\begin{equation}
    \left\{
    \begin{array}{ll}
        \mathrm{Rm}(g_0)+\e_1 g_0\owedge g_0\in \mathrm{C}_{\mathrm{PIC1}};\\
              \displaystyle \int^{1}_0 s\left( \fint_{B_{g_0}(x,s)}|\Rm(g_0)|\, d\mathrm{vol}_{g_0}\right)\, ds<\e_2
    \end{array}
    \right.
\end{equation}
for all $x\in M$, can be deformed to one which is almost flat.

We apply Theorem~\ref{thm:pseudo-locality} on $g_0$ with $r=\mathrm{diam}(M,g_0)=1$ with $\Lambda_0=\e_1$ so that there exists a Ricci flow $g(t),t\in [0,S_n]$ with $g(0)=g_0$ and $S_n<1$. We claim that if $\e_1$ and $\e_2$ are sufficiently small, then $\tilde g=g(S_n)$ will satisfy the almost flatness condition in  Gromov–Ruh Theorem \cite{Gromov1978,Ruh1982}. This follows by refining the proof of Theorem~\ref{thm:pseudo-locality}. If $\e_1$ is sufficiently small, we know that $\Ric(g(t))\geq -1$ along the flow and hence, 
\begin{equation}
    d_{g(t)}(x,y)\leq e^t d_{g_0}(x,y)
\end{equation}
for all $t\in [0,S_n]$ and $x,y\in M$. Hence, $\mathrm{diam}(M,g(t))\leq C_0(n)$. Since $\mathrm{Rm}(g(t))+L_n\e_1 g(t)\owedge g(t)\in \mathrm{C}_{\mathrm{PIC1}}$, it suffices to estimate $\mathcal{R}(g(t))$ from above by shrinking $\e_1$. Here we might assume $L_n$ to be large. As in the proof of Theorem~\ref{thm:pseudo-locality}, we consider the twisted curvature tensor $\widetilde{\Rm}=\Rm+2L_n\e_1 g(t)\owedge g(t)$ so that $\varphi=e^{-4L\e_1 n(n-1)t}\tilde{\mathcal{R}}$ satisfies $\rheat \varphi\leq \mathcal{R} \varphi$ and thus the same analysis as in the proof of Theorem~\ref{thm:pseudo-locality} (but simpler) yields 
\begin{equation}
    \begin{split}
        \varphi(x,t)&\leq  \int_M G(x,t;y,0) \tilde{\mathcal{R}}(y,0) \,d\mathrm{vol}_{g_0}(y)\leq  C_n (\e_1+\e_2) t^{-1}.
    \end{split}
\end{equation}

And hence, $\mathrm{diam}(M,\tilde g)^2|\Rm(\tilde g)|\leq C_n(\e_1+\e_2)S_n^{-1}+C_n \e_1$ for some dimensional constant $C_n>0$. By choosing $\e_1,\e_2>0$ both small enough so that $C_n(\e_1+\e_2)S_n^{-1}$ is smaller than the dimensional constant in Gromov–Ruh Theorem \cite{Gromov1978,Ruh1982}, we conclude that $M$ supports metric $\tilde g$ which is almost flat and hence $M$  is diffeomorphic to an infranil manifold. Result follows by relabelling the constants.
\end{proof}


\begin{thebibliography}{100}



\bibitem{BamlerCabezasWilking2019}Bamler, R; Cabezas-Rivas, E; Wilking, B., {\sl The Ricci flow under almost non-negative curvature conditions}. Invent. Math. 217 (2019), no. 1, 95--126.


\bibitem{Brendle2009}Brendle, S., {\sl A generalization of Hamilton’s differential Harnack inequality for the Ricci flow}. J. Differential Geom. 82 (2009), no. 1, 207–227.

\bibitem{BrendleBook}Brendle, S., {\sl Brendle, Simon. Ricci flow and the sphere theorem.} Graduate Studies in Mathematics, 111. American Mathematical Society, Providence, RI, 2010. viii+176 pp. ISBN: 978-0-8218-4938-5

\bibitem{Bryant2005}Bryant, Robert, {\sl Ricci flow solitons in dimension three with  SO(3)-symmetries}.



\bibitem{CW}Cabezas-Rivas, E.; Wilking, B., {\sl How to produce a Ricci flow via Cheeger-Gromoll exhaustion}. J. Eur. Math. Soc. (JEMS) 17 (2015), no. 12, 3153–3194.
 




\bibitem{ChauTamYu2011}Chau, A.; Tam,  L.-F.; Yu,  C. , {\sl Pseudolocality for the Ricci flow and applications}, Canad. J. Math. Vol. 63(1), 2011, 55--85.

\bibitem{CheegerColding1997} Cheeger, J.; Colding, T.H., {\sl On the structure of spaces with Ricci curvature bounded below. I}, J. Differential Geom. 46 (1997), no. 3, 406–480.



\bibitem{ChenZhu2002} Chen, B.-L.; Zhu, X.-P.,   {\sl A gap theorem for complete noncompact manifolds with nonnegative Ricci curvature.} Comm. Anal. Geom. 10 (2002), no. 1, 217--239. 

\bibitem{ChenZhu2003}Chen, B.-L.; Zhu, X.-P.,  {\sl On complete noncompact K\"ahler manifolds with positive bisectional curvature}. Math. Ann. 327 (2003), no. 1, 1--23.



\bibitem{Chen2009} Chen, B.-L., {\sl Strong uniqueness of the Ricci flow}, J. Differential Geom. \textbf{82} (2009), no. 2, 363--382, MR2520796, Zbl 1177.53036.


\bibitem{ChenWeiYe2022} Chen, E.;  Wei, G.; Ye, R., {\sl Ricci Flow and Gromov Almost Flat Manifolds}, arXiv:2203.05107.

\bibitem{ChowBookI}  Chow, B; Chu, S.-C.; Glickenstein, D.; Guenther, C.; Isenberg, J.; Ivey, T.; Knopf, D.; Lu, P.; Luo, F.; Ni, L., {\sl  The Ricci flow: Techniques and Applications. Part I. Geometric aspects.}. ‘Mathematical Surveys and Monographs,’ \textbf{135} A.M.S. 2007.

\bibitem{ChowBookII}  Chow, B; Chu, S.-C.; Glickenstein, D.; Guenther, C.; Isenberg, J.; Ivey, T.; Knopf, D.; Lu, P.; Luo, F.; Ni, L., {\sl  Ricci flow: Techniques and Applications: Part II: Analytic aspects}. ‘Mathematical Surveys and Monographs,’ \textbf{144} A.M.S. 2008.


\bibitem{Dress} Drees, G.,  {\sl Asymptotically flat manifolds of nonnegative curvature}, Differential Geom. Appl. 4 (1994), no. 1, 77--90.


\bibitem{EschenburgSchroederStrake1989} Eschenburg, J.; Schroeder, V.; Strake,  M., {\sl Curvature at infinity of open nonnegatively curved manifolds}, J. Differential Geom. 30 (1989), 155--166.



\bibitem{GreeneWu1973} Greene, R. E.;  Wu, H., {\sl On the subharmonicity and plurisubharmonicity of geodesically convex functions}, Indiana Univ. Math. J. 22 (1973), 641–653.


\bibitem{GreeneWu1974}Greene, R. E.;  Wu, H., {\sl Integrals of subharmonic functions on manifolds of nonnegative curvature}, Invent. Math. 27 (1974), 265–298.

\bibitem{GreeneWu1979}Greene, R. E.;  Wu, H.,  {\sl $C^\infty$ approximations of convex, subharmonic, and plurisubharmonic functions}, Ann. Sci. \'Ecole Norm. Sup. 12 (1979), 47–84.

\bibitem{GreeneWu1982} Greene, R. E.;  Wu, H.,  {\sl Gap theorems for noncompact Riemannian manifolds}, Duke Math. J. 49 (1982), 731--756.

\bibitem{GreenePetersenZhu1994}Greene, R. E.; Petersen,  P.,;  Zhu, S.,  {\sl Riemannian manifolds of faster-than-quadratic curvature decay}. Internat. Math. Res. Notices 1994, no. 9, 363ff., approx. 16 pp.

\bibitem{Gromov1978} Gromov, M., {\sl Almost flat manifolds}. J. Differential Geometry, 13(2):231–241, 1978.


\bibitem{Guenther2002}Guenther, C.-M. , {\sl The fundamental solution on manifolds with time-dependent metrics}. J. Geom. Anal. 12(2002), no. 3, 425--436.

\bibitem{Hamilton1995}Hamilton R.-S.,  {\sl A compactness property for solutions of the Ricci flow}, American J. Math. 117 (1995) 545--572.

\bibitem{HeLee2021}He, F.; Lee, M.-C., {\sl Weakly PIC1 manifolds with maximal volume growth}. J. Geom. Anal. 31 (2021), no. 11, 10868--10885.

\bibitem{Hochard}Hochard, R., {\sl Short-time existence of the Ricci flow on complete, non-collapsed 3- manifolds with Ricci curvature bounded from below}, preprint, arXiv:1603.08726.

\bibitem{Lai2019}Lai, Y., {\sl Ricci flow under local almost non-negative curvature conditions}. Adv. Math. 343 (2019), 353--392. 

\bibitem{Lai2022} Lai, Y., {\sl O(2)-symmetry of 3D steady gradient Ricci solitons}, arXiv:2205.01146.


\bibitem{LeeTam2020}Lee, M.-C.; Tam, L.-F., {\sl Chern-Ricci flows on noncompact manifolds}, J. Differential Geom. 115 (2020), no. 3, 529–564.

\bibitem{LeeTam2022}Lee, M.-C.; Tam, L.-F., {\sl Some local maximum principles along Ricci flows}. Canad. J. Math. 74 (2022), no. 2, 329--348.


\bibitem{Li2016} Li, M.,   {\sl Gap theorems for locally conformally flat manifolds}. J. Differential Equations 260 (2016), no. 2, 1414--1429. 


\bibitem{Liu2013}Liu, G., {\sl 3-manifolds with nonnegative Ricci curvature}. Invent. Math. 193 (2013), no. 2, 367--375.

\bibitem{Lott2019}Lott, J., {\sl On 3-manifolds with pointwise pinched nonnegative Ricci curvature}. Math. Ann. (2023). https://doi.org/10.1007/s00208-023-02596-9

\bibitem{MantegazzaMascellaniUraltsev2014} Mantegazza, C.; Mascellani, G.; Uraltsev, G., {\sl On the distributional Hessian of the distance function.} Pacific J. Math. 270 (2014), no. 1, 151–166.

\bibitem{MokSiuYau1981} Mok, N. ;  Siu, Y.-T.;  Yau, S.-T.,  {\sl The Poincar\'e-Lelong equation on complete K\"ahler manifolds}. Compositio Math. 44(1981), 183--218.


\bibitem{MunteanuSungWang2019} Munteanu, O.; Sung, A. C.-J.; Wang, J., {\sl  Poisson equation on complete manifolds}. Adv. Math. 348 (2019), 81–145.


\bibitem{NiShiTam}Ni, L.; Shi, Y.; Tam, L.-F., {\sl Poisson equation, Poincar\'e-Lelong equation and curvature decay on complete K\"ahler manifolds}. J. Differential Geom. 57 (2001), no. 2, 339--388.


\bibitem{NiTam}Ni, L.; Tam, L.-F., {\sl Plurisubharmonic functions and the structure of complete K\"ahler manifolds with nonnegative curvature}. J. Differential Geom. 64 (2003), no. 3, 457--524.

\bibitem{Ni2012} Ni, L., {\sl An optimal gap theorem}, Invent. Math. 189 (2012) 737--761.

\bibitem{NiNiu2020}Ni, L.; Niu, Y.,  {\sl Gap theorem on K\"ahler manifolds with nonnegative orthogonal bisectional curvature}. J. Reine Angew. Math. 763 (2020), 111--127.



\bibitem{Perelman2002}Perelman, G., {\sl The entropy formula for the Ricci flow and its geometric applications}, arXiv:math.DG/0211159.



\bibitem{Reiris2015}Reiris, M., {\sl On Ricci curvature and volume growth in dimension three}. J. Differential Geom. 99 (2015), no. 2, 313--357.

\bibitem{Rothaus1981}Rothaus, O. S.,  {\sl Logarithmic Sobolev inequalities and the spectrum of Schr\"odinger Operator}, Journal of functional analysis, 42 (1981), pp. 110--120.

\bibitem{Ruh1982}Ruh, E.A., {\sl Almost flat manifolds}. J. Differential Geometry, 17(1):1--14, 1982.

\bibitem{SchulzeSimon2013}Schulze, F.; Simon, M., {\sl Expanding solitons with non-negative curvature operator coming out of cones}. Math. Z. 275 (2013), no. 1-2, 625–639.



\bibitem{Shi1989}Shi, W.-X., {\sl Deforming the metric on complete Riemannian manifolds}. J. Differential Geom. 30 (1989), no. 1, 223--301.




\bibitem{Shi1997}Shi, W.-X., {\sl Ricci flow and the uniformization on complete noncompact K\"ahler manifolds}.  J. Differential Geom. 45 (1997), no. 1, 94--220.




\bibitem{Simon2008}Simon, M., {\sl Local results for flows whose speed or height is bounded by $c/t$}, Int. Math. Res. Not. IMRN 2008, Art. ID rnn 097, 14 pp, MR2439551, Zbl 1163.53042.



\bibitem{SimonTopping2016} Simon, M.; Topping, P. M., {\sl Local control on the geometry in 3D Ricci flow,} J. Differential Geom. 122 (2022), no. 3, 467--518. 

\bibitem{SimonTopping2017}Simon, M.; P.-M. Topping., {\sl Local mollification of Riemannian metrics using Ricci flow, and Ricci limit spaces}, Geom. Topol. {\bf 25} (2021), no. 2, 913--948.



\bibitem{TianZhang2021}Tian, G.; Zhang, Z. , {\sl Relative volume comparison of Ricci flow.} Sci. China Math. 64, 1937--1950 (2021). https://doi.org/10.1007/s11425-021-1869-5

\bibitem{Topping2010}Topping, P. M., {\sl Ricci flow compactness via pseudolocality, and flows with incomplete initial metrics}, J. Eur. Math. Soc. (JEMS) 12 (2010), no. 6, 1429–1451.

\bibitem{Wang2018}Wang, B., {\sl The local entropy along Ricci flow Part A: the no-local-collapsing theorems}. Camb. J. Math. {\bf 6} (2018), no. 3, 267-346.




\bibitem{Yokota2008}Yokota, T., {\sl Curvature integrals under the Ricci flow on surfaces}. Geom. Dedicata 133 (2008), 169–179.


\end{thebibliography}
\end{document}